\documentclass[reqno]{amsart}

\usepackage{amsfonts, amsmath, amsthm, amstext, amssymb}
\usepackage{url}
\usepackage{paralist}
\usepackage{tikz}
\usepackage[english]{babel}
\usepackage{mathtools}
\usepackage{framed}
\usepackage{amsrefs}

\newtheorem{lemma}{Lemma}[section]
\newtheorem{theorem}[lemma]{Theorem}
\newtheorem{conjecture}[lemma]{Conjecture}

\newtheorem{corollary}[lemma]{Corollary}

\theoremstyle{definition}\newtheorem{definition}[lemma]{Definition}
\theoremstyle{definition}\newtheorem{remark}[lemma]{Remark}
\theoremstyle{definition}

\begin{document}

\title{Hyberbolic Belyi maps and Shabat-Blaschke products}

\author{Kenneth Chung Tak Chiu}
\address{Department of Mathematics,  University of Toronto, Toronto, Canada.}
\email{kennethct.chiu@mail.utoronto.ca}
\thanks{}

\author{Tuen Wai Ng}
\address{Department of Mathematics,  The University of Hong Kong, Pokfulam, Hong Kong.}
\email{ntw@maths.hku.hk}



\date{June 29, 2019}

\dedicatory{}

\begin{abstract}
We first introduce hyperbolic analogues of Belyi maps, Shabat polynomials and Grothendieck's dessins d'enfant. In particular we introduce and study the Shabat-Blaschke products and the size of their hyperbolic dessin d'enfants in the unit disk. We then study a special class of Shabat-Blaschke products, namely the Chebyshev-Blaschke products. Inspired by the work of Ismail and Zhang (2007) on the coefficients of the Ramanujan's entire function, we will give similar arithmetic properties of the coefficients of the Chebyshev-Blaschke products and then use them to prove some Landen-type identities for theta functions.
\end{abstract}

\maketitle

\section{Introduction}\label{introduction}

It is well known that there is a bijective correspondence between the connected compact Riemann surfaces and the nonsingular irreducible complex projective curves \cite[p. 22-24]{Jones}. In 1979, G. V. Belyi proved the following theorem:

\begin{theorem}[Belyi, \cite{Belyi}]
A connected compact Riemann surface $X$ is defined over the field $\overline{\mathbb{Q}}$ of algebraic numbers if and only if there exists a nonconstant holomorphic map $f:X\rightarrow \widehat{\mathbb{C}}$ with at most $3$ critical values in the Riemann sphere $\widehat{\mathbb{C}}$. In such case, $f$ is isomorphic to a branched covering that is defined over $\overline{\mathbb{Q}}$.
\end{theorem}

A branched covering $X$ of the Riemann sphere $\widehat{\mathbb{C}}$ ramified over at most three points $a, b,c$ has then been called a \textbf{Belyi map}. Inspired by Belyi's theorem, Grothendieck  introduced in 1984 the theory of dessin d'enfant in his \emph{Esquisse d'un programme}  \cite{Grothendieck} in the hope of a better understanding of the absolute Galois group $Gal(\overline{\mathbb{Q}}/\mathbb{Q})$. The \textbf{dessin d'enfant} of a Belyi map has been defined to be the preimage under the Belyi map of the geodesic between $a$ and $b$. \\

Let $F_2=\langle g_1, g_2\rangle$ be the free group of rank $2$ and $S_n$ be the symmetric group acting on $\{1,\dots, n\}$. A \textbf{monodromy representation} is a group homomorphism $\rho: F_2\rightarrow S_n$. We say a monodromy representation  $\rho: F_2\rightarrow S_n$ is \textbf{transitive} if $\rho(F_2)$ acts on $\{1,\dots, n\}$ transitively. Two monodromy representations $\rho_1:F_2\rightarrow S_n$ and $\rho_2:F_2\rightarrow S_n$  are said to be \textbf{equivalent} if there exists a permutation $\iota$ on $\{1,\dots, n\}$ such that $\rho_1(g_i)\circ \iota=\iota\circ\rho_2(g_i)$ for each $i=1,2$. It is easy to check that this indeed defines an equivalence relation on the collection of all monodromy representations, and that if two monodromy representations are equivalent and one of them is transitive, then so is the other. The category of Belyi maps, the category of dessins d'enfant and the category of transitive monodromies are all equivalent to each other.  The detailed explanations can be found in \cite{Girondo} and \cite{Zvonkin}. In particular, one has

\begin{theorem}[{\cite[p. 148-155]{Girondo}}]\label{bijective correspondence Belyi maps onto sphere and monodromy}
There is a bijective correspondence between the equivalence classes of Belyi maps onto the Riemann sphere and the equivalence classes of transitive monodromy representations.
\end{theorem}

Roughly speaking,  monodromies and dessins d'enfant, which are discrete objects, determine uniquely the Belyi maps which are arithmetic objects. Moreover, Grothendieck introduced the Galois action of $Gal(\overline{\mathbb{Q}}/\mathbb{Q})$ on the Belyi maps, and hence on the dessins d'enfant \cite{Grothendieck}.

\begin{definition}
Let $F_2:=\langle g_1,g_2\rangle$ be the rank $2$ free group. Given a transitive monodromy representation $\rho:F_2\rightarrow S_n$. Denote the numbers of cycles in $\rho(g_1)$ and $\rho(g_2)$ by $c_1$ and $c_2$ respectively. We say that $\rho$ is a \textbf{tree} if $c_1+c_2=n+1$.
\end{definition}

It is easy to check that if $\rho$ and $\rho^\prime$ are equivalent transitive monodromy representations, then $\rho(g_i)$ and $\rho^\prime(g_i)$ have the same number of cycles, for $i=1,2$. Hence if one of $\rho$ and $\rho^\prime$ is a tree, then so is the other. A polynomial with at most two finite critical values is called a \textbf{Shabat polynomial}, which is clearly a Belyi map. The following subcorrespondence was proved by Shabat and Voevodsky \cite{Shabat}:

\begin{theorem}[{\cite[p. 84-85]{Zvonkin}, \cite{Shabat}}]\label{trees and Shabat polynomials}
An equivalence class of transitive monodromy representations is a tree if and only if the corresponding equivalence class of Belyi maps consists of a Shabat polynomial.
\end{theorem}

There is also a Galois action of $Gal(\overline{\mathbb{Q}}/\mathbb{Q})$ on the Shabat polynomials, and hence on the trees. It was proved by Lenstra and Schneps \cite{Schneps} that this action is faithful. Following Grothendieck, one hopes that the structures of $Gal(\overline{\mathbb{Q}}/\mathbb{Q})$ can be revealed from the combinatorial properties of the trees or the dessins d'enfants.\\

We will establish a hyperbolic analogue (Theorem \ref{hyperbolic Belyi and transitive monodromy}) of Grothendieck's theory of dessins d'enfant by replacing the Riemann sphere $\widehat{\mathbb{C}}$ with three marked points by the open unit disk $\mathbb{D}$ with two marked points. However, in this analogue the \textbf{hyperbolic Belyi maps} constructed from a given transitive monodromy are not rigid, i.e. they depend on the hyperbolic distance between the two marked points under the Poincar\'{e} metric. Moreover, we will establish a hyperbolic analogue (Theorem \ref{Shabat-Blaschke products and trees}) of Shabat's correspondence. Indeed, the tree monodromies will correspond to finite Blaschke products with at most two critical values in $\mathbb{D}$, and such finite Blaschke products will be called \textbf{Shabat-Blaschke products}. Again, Shabat-Blaschke products constructed from a tree monodromy are not rigid and they depend on the hyperbolic distance between the two critical values in $\mathbb{D}$. We will also introduce and study the size of the hyperbolic dessin d'enfant of a Shabat-Blaschke product in Section \ref{dessin}.\\

It is natural to ask if there is a hyperbolic analogue of Belyi's theorem when one replaces the Riemann sphere $\widehat{\mathbb{C}}$ by the open unit disk $\mathbb{D}$ and $X$ is a noncompact topologically finite Riemann surface. To formulate such a result, there are two problems one needs to address first: i) What should be the algebraic object associated with the noncompact topologically finite Riemann surface $X$? ii) What should be used to replace $\overline{\mathbb{Q}}$? We do not know how to answer the first question, except some speculation given in Section \ref{epilogue}. For the second question, it would be helpful to first study some concrete examples, in particular, the case $X=\mathbb{D}$ and the hyperbolic Belyi maps are Shabat-Blaschke products. It is known that the \textbf{Chebyshev polynomials} are examples of Shabat polynomials and their coefficients are integers. We will prove a hyperbolic analogue of this statement. \textbf{Chebyshev-Blaschke products}, which are hyperbolic analogues of Chebyshev polynomials, were studied by Ng, Tsang and Wang \cite{Tsang}\cite{Tsang1}\cite{WangNg}. The Chebyshev-Blaschke products are examples of Shabat-Blaschke products. We will first recall the definition of Chebyshev-Blaschke products. Then inspired by the work of Ismail and Zhang (\cite{Ismail}) on the coefficients of the Ramanujan's entire function, we will prove that the Chebyshev-Blaschke products are defined over 
$$\mathbb{Z}\left[\sqrt{k},\sqrt{k\circ s_n}, \frac{\omega_1\circ s_n}{\omega_1}\right]\subseteq \overline{\mathbb{Q}(j)},$$
where $n$ is the degree of the Chebyshev-Blaschke product, $s_n$ is the scaling by $n$, $k$ and $\omega_1$ are defined in terms of Jacobi theta functions, and $j$ is the $j$-invariant. This leads us to eventually show that the Chebyshev-Blaschke products are defined over $\mathbb{Z}[[q^{1/4}]]$, the ring of power series in $q^{1/4}$ over $\mathbb{Z}$, where $q=e^{2\pi i\tau}$. Finally,  we also obtain a family of Landen-type identities for theta functions as byproducts, which can degenerate to a family of trigonometric identities.\\

\section{Preliminaries}
Many properties of hyperbolic Belyi maps are topological, so we first recall some well-known results in topology that we are going to use.

\begin{lemma}[Homotopy lifting property, {\cite[p. 60]{Hatcher}}]\label{homotopy lifting property}
Suppose $X$, $Y$ and $Z$ are topological spaces,  $p:X\rightarrow Y$ is a  topological covering, and $f_t: Z\rightarrow Y$, $t\in[0,1]$, is a homotopy. Let $g: Z\rightarrow X$ be a continuous map such that $p\circ g=f_0$. Then there exists a unique homotopy $\widetilde{f}_t:Z\rightarrow X$ such that $\widetilde{f}_0=g$ and $p\circ \widetilde{f}_t=f_t$ for all $t\in [0,1]$.
\end{lemma}

In particular, if $Z$ is a point, then we have the following:

\begin{lemma}[Path lifting property]\label{path lifting property}
Suppose $X$ and $Y$ are topological spaces and $p: X\rightarrow Y$ is a topological covering. Let $\gamma:[0,1]\rightarrow Y$ be a continuous path in $Y$. For any $x\in X$ with $p(x)=\gamma(0)$, there exists a unique continuous path $\widetilde{\gamma}_x:[0,1]\rightarrow X$ such that $\widetilde{\gamma}_x(0)=x$ and $p\circ\widetilde{\gamma}_x=\gamma$.
\end{lemma}

\begin{lemma}\label{path lifting of closed path}
Suppose $X$ and $Y$ are topological spaces and $p: X\rightarrow Y$ is a finite topological covering of degree $n$. Let $\gamma:[0,1]\rightarrow Y$ be a continuous path in $Y$. Let $E:=p^{-1}(\gamma(0))=\{x_1,\dots,x_n\}$ and $F:=p^{-1}(\gamma(1))$. Define $\sigma:E\rightarrow F$ by  $\sigma(x_i)=\widetilde{\gamma}_{x_i}(1)$ for each $i$, where $\widetilde{\gamma}_{x_i}$ is defined in Lemma \ref{path lifting property}. Then $\sigma$ is bijective.
\end{lemma}

\begin{proof}
Since $E$ and $F$ have the same finite cardinality, it suffices to show that $\sigma$ is injective. Let the reversed path of $\gamma$ be defined by $\gamma^{-1}(t):=\gamma(1-t)$ for all $t\in [0,1]$. For each $i$ and each $t\in [0,1]$, let $\widetilde{\gamma}^{-1}_{x_i}(t):=\widetilde{\gamma}_{x_i}(1-t)$, we have $p\circ\widetilde{\gamma}^{-1}_{x_i}(t)=p(\widetilde{\gamma}_{x_i}(1-t))=\gamma(1-t)=\gamma^{-1}(t)$, so $p\circ\widetilde{\gamma}^{-1}_{x_i}=\gamma^{-1}$ for each $i$. Suppose $x_i, x_j\in E$ and $\sigma(x_i)=\sigma(x_j)$.  Then $\widetilde{\gamma}^{-1}_{x_i}$ and $\widetilde{\gamma}^{-1}_{x_j}$ are liftings of $\gamma^{-1}$, and both start at the point $\widetilde{\gamma}_{x_i}(1)=\widetilde{\gamma}_{x_j}(1)$. By the uniqueness in Lemma \ref{path lifting property}, $\widetilde{\gamma}^{-1}_{x_i}=\widetilde{\gamma}^{-1}_{x_j}$. In particular, $x_i=\widetilde{\gamma}_{x_i}(0)=\widetilde{\gamma}^{-1}_{x_i}(1)=\widetilde{\gamma}^{-1}_{x_j}(1)=\widetilde{\gamma}_{x_j}(0)=x_j$.
\end{proof}

\begin{lemma}\label{liftings of homotopic paths}
Suppose $X$ and $Y$ are topological spaces, $p:X\rightarrow Y$ is a topological covering, and $f_t:[0,1]\rightarrow Y$, $t\in [0,1]$, is a homotopy such that $f_t(0)=f_0(0)$ and $f_t(1)=f_0(1)$ for all $t\in[0,1]$. Suppose $g_0$ and  $g_1$ are  liftings of $f_0$ and $f_1$ respectively and $g_0(0)=g_1(0)$. Then $g_0$ and $g_1$ are homotopic and $g_0(1)=g_1(1)$. 
\end{lemma}

\begin{proof}
By Lemma \ref{homotopy lifting property}, there exists a homotopy $\widetilde{f}_t:[0,1]\rightarrow X$ such that $\widetilde{f}_0=g_0$ and $p\circ\widetilde{f_t}=f_t$ for all $t\in[0,1]$. Since the path $f_t(0)$ is a constant, by the uniqueness in Lemma \ref{path lifting property}, $\widetilde{f}_t(0)=\widetilde{f}_0(0)$ for all $t\in [0,1]$.  Similarly, $\widetilde{f}_t(1)=\widetilde{f}_0(1)$ for all $t\in [0,1]$. Now both $\widetilde{f}_1$ and $g_1$ are liftings of $f_1$, and $\widetilde{f}_1(0)=\widetilde{f}_0(0)=g_0(0)=g_1(0)$. By Lemma \ref{path lifting property}, $g_1=\widetilde{f}_1$, which is homotopic to $\widetilde{f}_0=g_0$. Moreover, $g_0(1)=\widetilde{f}_0(1)=\widetilde{f}_1(1)=g_1(1)$.
\end{proof}

\begin{lemma}[{\cite[p. 22]{Forster}}]\label{pullback of complex structure} Suppose $X$ is a connected Riemann surface, $Y$ is a Hausdorff topological space and $f:Y\rightarrow X$ is a local homeomorphism. Then there is a unique complex structure on $Y$ such that $f$ is holomorphic.
\end{lemma}

\begin{lemma}[{\cite[p. 29]{Forster}}]\label{unbranched covering from branched covering}
Suppose $X$ and $Y$ are connected Riemann surfaces, and $f:X\rightarrow Y$ is a nonconstant proper holomorphic mapping. Let $B$ be the set of all critical values of $f$, $X^\prime:=X\setminus f^{-1}(B)$ and $Y^\prime=Y\setminus B$. Then $f:X^\prime\rightarrow Y^\prime$ is an unbranched holomorphic covering.
\end{lemma}

\begin{lemma}[{\cite[p. 51]{Forster}}]\label{branched covering from unbranched covering}
Suppose $X$ is a Riemann surface, $A\subset X$ is a closed discrete subset. Let $X^\prime:=X\setminus A$. If $Y^\prime$ is a Riemann surface and $f^\prime: Y^\prime\rightarrow X^\prime$ is an unbranched holomorphic covering, then $f^\prime$ can be extended to a branched covering of $X$, i.e. there exists a Riemann surface $Y$, a nonconstant proper holomorphic mapping $f:Y\rightarrow X$, and a biholomorphism $\phi: Y\setminus f^{-1}(A)\rightarrow Y^\prime$ such that $f|_{Y\setminus f^{-1}(A)}=f^\prime\circ \phi$.
\end{lemma}

\begin{lemma}[{\cite[p. 52]{Forster}}]\label{unbranched covering uniquely determine branched covering}
Suppose $X,Y$ and $Z$ are connected Riemann surfaces, $f:Y\rightarrow X$ and $g: Z\rightarrow X$ are nonconstant proper holomorphic mappings. Let $A\subset X$ be a closed discrete subset. Let $X^\prime:=X\setminus A, Y^\prime:=f^{-1}(X^\prime)$ and $Z^\prime:=g^{-1}(X^\prime)$. If $\phi^\prime: Y^\prime\rightarrow Z^\prime$ is a biholomorphism such that $g\circ\phi^\prime=f|_{Y^\prime}$, then $\phi^\prime$ can be extended to a biholomorphism $\phi: Y\rightarrow Z$ such that $g\circ \phi=f$.\\
\end{lemma}

Let $n$ be a positive integer. The \textbf{Hecke congruence subgroup of level $n$} is defined to be
$$\Gamma_0(n)
:=\left\{\left(
\begin{array}{cc}
a & b
\\c & d
\end{array}
\right)\in SL(2,\mathbb{Z})
: c\equiv 0\text{ (mod } n)
\right\}.$$
Let $\mathbb{H}$  be the open upper half plane, and $k$ be a positive even integer.
A function $f:\mathbb{H}\rightarrow \mathbb{C}$ is said to be a \textbf{modular form of weight $k$ and of level $n$},  if all of the following conditions hold:
\begin{enumerate}
\item $f$ is holomorphic;
\item For any 
$\left(
\begin{array}{cc}
a & b
\\c & d
\end{array}
\right)\in \Gamma_0(n) \text{ and }\tau\in\mathbb{H},
$
$$f\left(\frac{a\tau+b}{c\tau+d}\right)=(c\tau+d)^kf(\tau);$$
\item $f$ is holomorphic at all the cusps \cite[p. 16-17]{Diamond}.\\
\end{enumerate}

A function $f:\mathbb{H}\rightarrow \mathbb{C}$ satisfying condition (2) is said to be of \textbf{weight $k$ invariant under $\Gamma_0(n)$}.

\begin{lemma}[{\cite[p. 21, 24]{Diamond}}]\label{isogeny modular forms}
If $f:\mathbb{H}\rightarrow \mathbb{C}$ is a modular form of weight $k$ and of level $n$, then $f(m\tau)$ is a modular form of weight $k$ and of level $mn$.
\end{lemma}

A function $f:\mathbb{H}\rightarrow \mathbb{C}\cup \{\infty\}$ is said to be a \textbf{modular function of level $n$} if all of the following conditions hold:
\begin{enumerate}
\item $f$ is meromorphic,
\item $f$ is invariant (i.e. weight $0$ invariant) under $\Gamma_0(n)$,
\item $f$ is meromorphic at all the cusps.\\
\end{enumerate}

For any positive integer, let $j$ to be the $j$-invariant and $j_n(\tau)=j(n\tau)$. We have the following:

\begin{lemma}[{\cite[p. 229]{Cox}}]\label{rational Fourier coefficients}
Let $f:\mathbb{H}\rightarrow \mathbb{C}$ be a modular function of level $n$ whose Fourier expansion at $\infty$ has rational coefficients. Then $f\in\mathbb{Q}(j,j_n)$.
\end{lemma}

\begin{lemma}[{\cite[p. 210]{Cox}}]\label{modular equations}
For any positive integer $n$, there exists a nonconstant polynomial $\Phi_n\in \mathbb{Q}[X,Y]$ such that $\Phi_n(j,j_n)=0$.\\
\end{lemma}

\section{Hyperbolic Belyi maps}\label{Hyperbolic Belyi maps}

A domain $U\subset\mathbb{C}$ is said to be \textbf{$n$-connected} if $\widehat{\mathbb{C}}\setminus U$ has $n+1$ components. A simply connected domain is a  $0$-connected domain. A \textbf{doubly connected domain} is a $1$-connected domain. A domain is \textbf{finitely connected} if it is $n$-connected for some $n\geq 0$. The open unit disk is denoted by $\mathbb{D}$.\\

Each doubly connected domain $U$ in $\mathbb{C}$ can be mapped conformally onto an annulus $A_r:=\{z: r<|z|<1\}$, for some $r\in [0,1)$, or to $\widehat{\mathbb{C}}\setminus \{0,\infty\}$.  Moreover, two annuli $A_{r_1}$ and $A_{r_2}$ are conformally equivalent if and only if $r_1=r_2$, and none of the annuli is conformally equivalent to $\widehat{\mathbb{C}}\setminus \{0,\infty\}$, see for example \cite[p. 96]{Conway} and \cite[p. 283]{Big rudin}. The modulus of $U$, denoted by $M(U)$, is defined to be $\frac{1}{2\pi}\log (1/r)$ if $0<r<1$, and to be $+\infty$ if $r=0$, when $U$ is conformally equivalent to $A_r$. Also, $M(U)$ is defined to be  $0$ if $U$ is conformally equivalent to $\widehat{\mathbb{C}}\setminus \{0,\infty\}$.
For each $t\in (0,1)$, the \textbf{Gr\"{o}tzsch's ring domain} $\mathfrak{G}_t$ is the doubly connected domain $\mathbb{D}-[0,t]$. The modulus of the Gr\"{o}tzsch's ring domain, $M(\mathfrak{G}_t)$, is a strictly decreasing continuous function in $t$ that maps onto $(0,+\infty)$ \cite[p. 59-62]{Lehto}. Suppose $l$ is a geodesic between two distinct points $a,b\in\mathbb{D}$ equipped with the Poincar\'{e} metric \cite{BM07}. Then $\mathbb{D}-l$ is conformal to a Gr\"{o}tzsch's ring domain. The modulus $M(\mathbb{D}-l)$ is uniquely determined by the hyperbolic length of $l$. Hence,  if $l_1$ and $l_2$ are hyperbolic line segments of different lengths, then $M(\mathbb{D}-l_1)\neq M(\mathbb{D}-l_2)$.\\

 A \textbf{hyperbolic Belyi map} is a tuple $(X,f,a,b)$, where $X$ is a connected Riemann surface, $a,b$ are two distinct points in the standard unit disk $\mathbb{D}$,  and $f:X\rightarrow\mathbb{D}$ is a nonconstant proper holomorphic mapping whose critical values all lie in $\{a,b\}$. The \textbf{modulus} of a hyperbolic Belyi map $(X, f,a,b)$ is defined to be  $M(\mathbb{D}-l)$, where $l$ is the geodesic between $a$ and $b$ under the Poincar\'{e} metric. Two hyperbolic Belyi maps $(X_1,f_1,a_1,b_1)$ and $(X_2,f_2,a_2,b_2)$ are said to be \textbf{equivalent} if there exists $\phi\in Aut(\mathbb{D})$ and biholomorphism $\varphi:X_1\rightarrow X_2$ such that $\phi(a_1)=a_2$, $\phi(b_1)=b_2$, and $f_2\circ\varphi=\phi\circ f_1$. It is easy to check that this indeed defines an equivalence relation on the collection of all hyperbolic Belyi maps.\\

Given a hyperbolic Belyi map $(X,f,a,b)$,  we can associate a transitive monodromy representation to it in the following way:\\

Let $y\in \mathbb{D}\setminus\{a,b\}$. Let $\gamma_1,\gamma_2:[0,1]\rightarrow \mathbb{D}\setminus \{a,b\}$ be some continuous paths  both starting and ending at $y$ such that $\gamma_1$ is homotopic to an anticlockwise small circle around $a$ that separates the two points $a$ and $b$, and $\gamma_2$ is homotopic to an anticlockwise small circle around $b$ that separates the two points $a$ and $b$. Let $n:=\deg f$ and $E:=f^{-1}(y)=\{x_1,\dots, x_n\}$. By Lemma \ref{unbranched covering from branched covering} and the path lifting property (Lemma \ref{path lifting property}), $\gamma_1$ is lifted to $n$ paths $\widetilde{\gamma_1}_{x_1},\dots, \widetilde{\gamma_1}_{x_n}$. Similarly for the path $\gamma_2$. Now define $\sigma_1:E\rightarrow E$ by  $\sigma_1(x_i)=\widetilde{\gamma_1}_{x_i}(1)$ for each $i=1,\dots, n$, and $\sigma_2:E\rightarrow E$ by $\sigma_2(x_i)=\widetilde{\gamma_2}_{x_i}(1)$ for each $i=1,\dots, n$. Then by Lemma \ref{path lifting of closed path},  $\sigma_1$ and $\sigma_2$ are bijections on $E$. Let $\delta: E\rightarrow \{1,\dots, n\}$ be a bijection. We define the group homomorphism $\rho:F_2\rightarrow S_n$ by $g_i\mapsto \delta\circ \sigma_i\circ\delta^{-1}$, $i=1,2$. It is easy to show that this construction of $\rho$ is up to equivalence  independent of the choice of the bijection $\delta$. By Lemma \ref{liftings of homotopic paths}, $\sigma_1,\sigma_2$ and hence $\rho$ are independent of the choices of $\gamma_1$ and $\gamma_2$. Suppose $x$ is another point in $\mathbb{D}\setminus \{a,b\}$, and suppose $\beta_1,\beta_2$ are homotopic to $\gamma_1,\gamma_2$ respectively and they start and end at $x$. Let $D:=f^{-1}(x)$. Define the bijections $\lambda_1,\lambda_2:D\rightarrow D$ from $\beta_1,\beta_2$ in the same way as we define $\sigma_1,\sigma_2$ from $\gamma_1,\gamma_2$. There exists a continuous path $\alpha:[0,1]\rightarrow \mathbb{D}\setminus \{a,b\}$ such that $\alpha(0)=x$ and $\alpha(1)=y$.  Define $\eta:D\rightarrow E$ by sending a point $p$ in $D$ to the endpoint of the lifting of $\alpha$ starting at $p$. By Lemma \ref{path lifting of closed path},  $\eta$ is a bijection. By Lemma \ref{liftings of homotopic paths}, we have $\lambda_i=\eta^{-1}\circ\sigma_i\circ\eta$ for $i=1,2$. Therefore, the construction of $\rho$ is independent of the choice of the base point $y$.\\

Next, we want to prove that $\rho$ is transitive. Suppose $x_i,x_j\in E$. Since $X$ is a connected Riemann surface, $X$ is path connected. Since $X$ is a path connected Riemann surface and $f^{-1}(\{a,b\})$ is a finite set, $X\setminus f^{-1}(\{a,b\})$ is still path connected. So there exists a path $\beta:[0,1]\rightarrow X\setminus f^{-1}(\{a,b\})$ such that $\beta(0)=x_i$ and $\beta(1)=x_j$. Then $f\circ \beta$ is a closed path that starts and ends at $y$. Since the fundamental group $\pi_1(\mathbb{D}\setminus\{a,b\}, y)$ is generated by $\gamma_1$ and $\gamma_2$, we have that $f\circ\beta$ is homotopic (with endpoint $y$ fixed) to $\gamma_{m_1}\cdots\gamma_{m_k}$, where $m_1,\dots, m_k=1,2$. 
Define $$p_1:=\widetilde{\gamma_{m_1}}_{x_i}(1), ~ p_2:=\widetilde{\gamma_{m_2}}_{p_1}(1),~\dots~,~ p_k:=\widetilde{\gamma_{m_k}}_{p_{k-1}}(1). $$
Then 
\begin{eqnarray}\nonumber
f\circ (\widetilde{\gamma_{m_1}}_{x_i}\widetilde{\gamma_{m_2}}_{p_1}\cdots \widetilde{\gamma_{m_k}}_{p_{k-1}})
&=&(f\circ\widetilde{\gamma_{m_1}}_{x_i})(f\circ\widetilde{\gamma_{m_2}}_{p_1})\cdots (f\circ\widetilde{\gamma_{m_k}}_{p_{k-1}})
\\&=&\gamma_{m_1}\gamma_{m_2}\cdots\gamma_{m_k}.\nonumber
\end{eqnarray}
By Lemma \ref{liftings of homotopic paths}, $\widetilde{\gamma_{m_1}}_{x_i}\widetilde{\gamma_{m_2}}_{p_1}\cdots \widetilde{\gamma_{m_k}}_{p_{k-1}}$ ends at $x_j$.
Define $\sigma:=\sigma_{m_k}\circ\dots\circ\sigma_{m_1}$. Then $\sigma(x_i)=p_k=x_j$, so $\rho$ is transitive.

\begin{lemma}\label{well-definedness}
Suppose $(X,f,a,b)$ and $(X^\prime, f^\prime,a^\prime,b^\prime)$ are two equivalent hyperbolic Belyi maps. Then the transitive monodromy representations $\rho$ and $\rho^\prime$ associated respectively to these two hyperbolic Belyi maps are equivalent.
\end{lemma}

\begin{proof}
Let $y, \gamma_1,\gamma_2, E, \sigma_1,\sigma_2$ be as above. Since $(X,f)$ and $(X^\prime, f^\prime)$ are equivalent, there exists a biholomorphism $\varphi:X\rightarrow X^\prime$ and $\phi\in Aut(\mathbb{D})$ such that $\phi(a)=a^\prime$, $\phi(b)=b^\prime$ and  $f^\prime\circ\varphi=\phi\circ f$. Let  $E^\prime:=f^{\prime -1}(\phi(y))$. Since winding numbers are invariant under biholomorphism, the closed paths $\phi\circ \gamma_1$ and $\phi\circ \gamma_2$ are representatives of generators of $\pi_1(\mathbb{D}\setminus \{a^\prime,b^\prime\})$. Let $\sigma_1^\prime, \sigma_2^\prime$ be the bijections on $E^\prime$ obtained respectively  from $\phi\circ \gamma_1$ and $\phi\circ \gamma_2$. It is easy to check that $\varphi(E)= E^\prime$, so $\varphi|_E:E\rightarrow E^\prime$ is a bijection. Let $x_i\in E$. Since $f^\prime\circ \varphi\circ\widetilde{\gamma_1}_{x_i}=\phi\circ\gamma_1$, we know that both $\widetilde{\phi\circ\gamma_1}_{\varphi(x_i)}$ and $\varphi\circ \widetilde{\gamma_1}_{x_i}$ are liftings of $\phi\circ\gamma_1$ by $f^\prime$, and both paths start from $\varphi(x_i)$. By Lemma \ref{path lifting property}, $\widetilde{\phi\circ\gamma_1}_{\varphi(x_i)}(1)=\varphi\circ \widetilde{\gamma_1}_{x_i}(1)$, so $\sigma^\prime_1\circ \varphi|_E(x_i)=\varphi|_E\circ\sigma_1(x_i)$. This is true for each $x_i\in E$, so $\sigma_1^\prime\circ\varphi|_E=\varphi|_E\circ\sigma_1$. Similarly, $\sigma_2^\prime\circ\varphi|_E=\varphi|_E\circ\sigma_2$. 
\end{proof}

We also have the converse of the above lemma:

\begin{lemma}\label{injectivity}
If $(X,f,a,b)$ and $(X^\prime, f^\prime,a^\prime,b^\prime)$ are hyperbolic Belyi maps of the same modulus whose associated transitive monodromy representations $\rho$ and $\rho^\prime$ are equivalent, then $(X,f,a,b)$ and $(X^\prime, f^\prime,a^\prime,b^\prime)$ are equivalent.
\end{lemma}

\begin{proof}
Let $y,E,\gamma_1,\gamma_2,\sigma_1,\sigma_2$ be the intermediate notations previously defined  corresponding to $(X,f,a,b)$. Let $y^\prime, E^\prime, \gamma_1^\prime,\gamma_2^\prime, \sigma_1^\prime, \sigma_2^\prime$ be that corresponding to  $(X^\prime, f^\prime,a^\prime,b^\prime)$. Then there exists a bijection $\varphi: E\rightarrow E^\prime$ such that $\sigma^\prime_i\circ \varphi=\varphi \circ \sigma_i$, for $i=1,2$. Let $x_i\in E$ and $(f|_{X\setminus f^{-1}\{a,b\}})_*$ be the induced group homomorphism of $f|_{X\setminus f^{-1}\{a,b\}}$ pointed at $x_i$. We know that
 $$(f|_{X\setminus f^{-1}\{a,b\}})_*(\pi_1(X\setminus f^{-1}\{a,b\},x_i))=\{\gamma\in\pi_1(\mathbb{D}\setminus\{a,b\},y)|\widetilde{\gamma}_{x_i}(1)=x_i\},$$
which is in turn equal to the preimage under the group homomorphism $\pi_1(\mathbb{D}\setminus \{a,b\},y)\rightarrow Bij(E)$ of the subgroup of bijections on $E$ fixing $x_i$. Since $\sigma^\prime_i\circ \varphi=\varphi \circ \sigma_i$, for $i=1,2$, this group is then isomorphic (via $\gamma_i\mapsto\gamma_i^{\prime}$) to the preimage under the group homomorphism $\pi_1(\mathbb{D}\setminus\{a^\prime,b^\prime\},y^\prime)\rightarrow Bij(E^\prime)$ of the subgroup of bijections on $E^\prime$ fixing $\varphi(x_i)$, which is equal to
$$
\begin{array}{ll}
&\{\gamma\in\pi_1(\mathbb{D}\setminus\{a^\prime,b^\prime\},y^\prime)|\widetilde{\gamma}_{\varphi(x_i)}(1)=\varphi(x_i)\}
\\=&(f^\prime|_{X^\prime \setminus f^{\prime -1}\{a^\prime, b^\prime\}})_*(\pi_1(X^\prime \setminus f^{\prime -1}\{a^\prime,b^\prime\},\varphi(x_i))).
\end{array}
$$
Since the two hyperbolic Belyi maps are of the same modulus, there exists $\phi\in Aut(\mathbb{D})$ such that $\phi(a)=a^\prime$ and $\phi(b)=b^\prime$. By Proposition 1.37 in \cite[p. 67]{Hatcher}, the two coverings 
$$X\setminus f^{-1}\{a,b\}\xrightarrow{f} \mathbb{D}\setminus\{a,b\}$$
and
$$X^\prime\setminus f^{\prime -1}\{a^\prime,b^\prime\}\xrightarrow{f^\prime} \mathbb{D}\setminus\{a^\prime,b^\prime\}\xrightarrow{\phi^{-1}} \mathbb{D}\setminus\{a,b\}$$
are isomorphic as topological coverings. By Lemma \ref{pullback of complex structure}, these two holomorphic coverings are isomorphic. By Lemma \ref{unbranched covering uniquely determine branched covering}, the two hyperbolic Belyi maps are equivalent.
\end{proof}

We also have the following variant of the Riemann existence theorem:
\begin{lemma}\label{surjectivity}
Given a $\lambda\in (0,+\infty)$ and a transitive monodromy representation $\rho$, there exists a hyperbolic Belyi map $(X,f,a,b)$ of modulus $\lambda$ whose associated transitive monodromy representation is equivalent to $\rho$.
\end{lemma}

\begin{proof}
Suppose the codomain of $\rho$ is $S_n$ for some positive integer $n$. Let $\Delta_1,\dots, \Delta_n$ be copies of the open unit disk slitted along $[0,1)$. Let $t\in (0,1)$ such that the modulus of $\mathbb{D}\setminus [0,t]$ is $\lambda$. Assume along the slit, each slitted disk $\Delta_i$ has two disjoint copies of  $(0,t)$ attached, named $\alpha_i^u$ and $\alpha_i^l$, and has two disjoint copies of $(t,1)$ attached, named $\beta_i^u$ and $\beta_i^l$. We construct a connected topological space $Z$ from $\sqcup_i \Delta_i$ by gluing the edges $\alpha_i^l$ and $\alpha_j^u$ together whenever $\rho(g_1)(i)=j$, and gluing the edges $\beta_i^l$ and $\beta_j^u$ together whenever $\rho(g_2)( \rho(g_1)(i))=j$. We define $g:Z\rightarrow \mathbb{D}\setminus\{0,t\}$ to be the canonical projection onto $\mathbb{D}\setminus\{0,t\}$, which is a topological covering. By Lemma \ref{pullback of complex structure}, there exists a complex structure on $Z$ such that $g: Z\rightarrow\mathbb{D}\setminus\{0,t\}$ is a holomorphic covering. By Lemma \ref{branched covering from unbranched covering}, $g$ extends to a branched covering $f:X\rightarrow\mathbb{D}$ whose critical values all lie in the set $\{0,t\}$. Now $f$ is the desired hyperbolic Belyi map.
\end{proof}

By Lemma \ref{well-definedness}, Lemma \ref{injectivity} and Lemma \ref{surjectivity}, we have the following:

\begin{theorem}\label{hyperbolic Belyi and transitive monodromy}
For each $\lambda\in (0,+\infty)$, there is a bijective correspondence between the equivalence classes of hyperbolic Belyi maps of modulus $\lambda$ and the equivalence classes of transitive monodromy representations.
\end{theorem}

This theorem is the hyperbolic analogue of Theorem \ref{bijective correspondence Belyi maps onto sphere and monodromy}. The main difference between the two versions is that in the hyperbolic analogue, we have the extra parameter $\lambda$ that depends on the hyperbolic distance between the two critical values in $\mathbb{D}$.\\

\section{Difference in Euler characteristics}

A Riemann surface is said to be \textbf{topologically finite} if it is homeomorphic to a closed surface with at most finitely many closed disks and points removed.

\begin{lemma}[Riemann-Hurwitz: topologically finite version, \cite{WangNg}\cite{Wang}]\label{topologically finite, Riemann-Hurwitz} 
Suppose $f:X\rightarrow Y$ is a nonconstant proper holomorphic mapping between Riemann surfaces. If $Y$ is topologically finite and $f$ has finitely many critical values, then $X$ is also topologically finite and the following Riemann-Hurwitz formula holds,
$$\deg R_f= \deg f \cdot \chi(Y)-\chi(X),$$
where $R_f$ is the ramification divisor of $f$, hence $\deg R_f$ is the sum of the order of the critical points of $f$, and $\chi(X)$ and $\chi(Y)$ are Euler characteristic of $X$ and $Y$ respectively.
\end{lemma}

By Lemma \ref{topologically finite, Riemann-Hurwitz}, given any hyperbolic Belyi map $(X,f,a,b)$, we know that $X$ is topologically finite. Moreover, since $\mathbb{D}$ is noncompact, and $f$ is continuous and surjective, we know $X$ is also noncompact, so $X$ is homeomorphic to a closed surface with at least a disk or a point removed.

\begin{lemma}\label{number of cycles}
Let $(X,f,a,b)$ be a hyperbolic Belyi map and $\sigma_1,\sigma_2$ be defined as in Section \ref{Hyperbolic Belyi maps}. Then the number of cycles of $\sigma_1$ equals the cardinality of $f^{-1}(a)$ and the number of cycles of $\sigma_2$ equals the cardinality of $f^{-1}(b)$.
\end{lemma}

\begin{proof}
By a theorem on the local behavior of a holomorphic mapping \cite[p. 10]{Forster}, we can choose a small circle $\beta$ around $a$ such that it is lifted by $f$ to $|f^{-1}(a)|$ cycles of paths, each cycle goes around a preimage of $a$ and the number of paths the cycle contains is equal to the multiplicity of that preimage. Next, we join the base point $y$ to the circle   $\beta$ by a path $\alpha$ in $\mathbb{D}\setminus\{a,b\}$. Then the closed path $\alpha\beta\alpha^{-1}$ is lifted to cycles of paths having the same property. Since $\alpha\beta\alpha^{-1}$ is homotopic to $\gamma_1$,  by Lemma \ref{liftings of homotopic paths}, the lifting of $\gamma_1$ is again  cycles of paths having the same property. Hence, the number of cycles of $\sigma_1$ equals $|f^{-1}(a)|$. Similarly, the number of cycles of $\sigma_2$ equals $|f^{-1}(b)|$.
\end{proof}

Fix a $\lambda\in (0,+\infty)$. Given a transitive monodromy representation $\rho:F_2:=\langle g_1,g_2\rangle\rightarrow S_n$. Let $(X,f,a,b)$ be the hyperbolic Belyi map  of modulus $\lambda$ associated to $\rho$. Let $(\overline{X},\overline{f},c,d,e)$ be the Belyi map onto the Riemann sphere associated to $\rho$. We have that $\deg f=n=\deg \overline{f}$. By Lemma \ref{number of cycles} and its analogue for Belyi maps onto the Riemann sphere, $|\overline{f}^{-1}(c)|=c_1=|f^{-1}(a)|$, similarly $|\overline{f}^{-1}(d)|=c_2=|f^{-1}(b)|$. By the Riemann-Hurwitz formula in Lemma \ref{topologically finite, Riemann-Hurwitz}, 
\begin{eqnarray}
\chi({\overline{X}})-\chi(X)
&=&\deg \overline{f}\cdot \chi(\overline{\mathbb{C}})-\deg R_{\overline{f}}-\deg f \cdot \chi(\mathbb{D})+\deg R_f\nonumber
\\&=&2\deg{\overline{f}}-\deg R_{\overline{f}}-\deg \overline{f}+\deg R_f\nonumber
\\&=&\deg{\overline{f}}-(3\deg\overline{f}-|\overline{f}^{-1}(c)|-|\overline{f}^{-1}(d)|-|\overline{f}^{-1}(e)|)\nonumber
\\&\quad&+(2\deg\overline{f}-|f^{-1}(a)|-|f^{-1}(b)|)\nonumber
\\&=&|\overline{f}^{-1}(e)|.\nonumber
\end{eqnarray}
Therefore, $\chi(X)$ is $\chi(\overline{X})$ minus the number of cycles in $\rho(g_1)^{-1}\circ \rho(g_2)^{-1}$, as $\gamma_c\gamma_d\gamma_e=1,$ where $\gamma_c,\gamma_d,\gamma_{e}$ are closed continuous paths in $\widehat{\mathbb{C}}$ with the same base point in $\widehat{\mathbb{C}}\setminus\{c,d,e\}$ that goes around $c$, $d$ and $e$ respectively, just like how $\gamma_1,\gamma_2$ goes around $a$ and $b$ in Section \ref{Hyperbolic Belyi maps}.

\begin{remark}
Alternatively, by comparing $X$ and $\overline{X}$ constructed using the cutting and pasting surgery in the Riemann Existence Theorem, one can see that $X$ differs from $\overline{X}$ by missing $c_3$ number of closed disks, where $c_3$ is the number of cycles in $\rho(g_1)^{-1}\circ \rho(g_2)^{-1}$.  Since taking away a disk from a surface  will decrease its Euler characteristic by $1$, we obtain the same formula as before. 
\end{remark}

When $\overline{X}=\widehat{\mathbb{C}}$, $\overline{f}$ is a nonconstant Shabat polynomial, and $e=\infty$, we have $|\overline{f}^{-1}(e)|=1$. Then $\chi(X)=\chi(\overline{X})-|\overline{f}^{-1}(e)|=2-1=1$. By Liouville's theorem, $X$ cannot be biholomorphic to $\mathbb{C}$, so $X$ is biholomorphic to $\mathbb{D}$. Therefore, in the next section we will study the hyperbolic Belyi maps $(X,f,a,b)$ when $X=\mathbb{D}$.\\

\section{Shabat-Blaschke products}

\begin{definition}
A \textbf{Shabat-Blaschke product} is a triple $(g,a,b)$, where $g:\mathbb{D}\rightarrow\mathbb{D}$ is a finite Blaschke product whose critical values all lie in $\{a,b\}$. The \textbf{modulus} of a Shabat-Blaschke product is defined to be $M(\mathbb{D}- l)$, where $l$ is the geodesic between $a$ and $b$. We say that two Shabat-Blaschke products $(g_1,a_1,b_1)$ and $(g_2,a_2,b_2)$ are \textbf{equivalent} if there exists $\phi,\varphi\in Aut(\mathbb{D})$ such that $\phi(a_1)=b_1$, $\phi(b_1)=b_2$, and $g_2\circ\varphi =\phi\circ g_1$.
\end{definition}

The following result gives a characterization of Shabat-Blaschke products:
\begin{theorem}\label{Shabat-Blaschke products and trees}
An equivalence class of hyperbolic Belyi maps consists of a hyperbolic Belyi map of the form $(\mathbb{D}, g,a,b)$, where $(g,a,b)$ is a  Shabat-Blaschke product, if and only if the corresponding equivalence class of transitive monodromy representations is a tree.
\end{theorem}

\begin{proof}
Let $(X,f,a,b)$ be a hyperbolic Belyi map and let $\rho:F_2\rightarrow S_n$ be its corresponding transitive monodromy representation. By the Riemann-Hurwitz formula in Lemma \ref{topologically finite, Riemann-Hurwitz} and by Lemma \ref{number of cycles}, 
\begin{eqnarray}
\chi(X)
&=&\deg f\cdot\chi(\mathbb{D})-\deg R_{f}\nonumber
\\&=&n-(2n-|f^{-1}(a)|-|f^{-1}(b)|)\nonumber
\\&=&-n+c_1+c_2.\nonumber
\end{eqnarray}
If $\rho$ is a tree, then $\chi(X)=1$, so $X$ is homeomorphic to the open unit disk $\mathbb{D}$. By Liouville's theorem, $X$ cannot be biholomorphic to $\mathbb{C}$. By the uniformization theorem, $X$ is thus biholomorphic to $\mathbb{D}$. Hence $(X,f,a,b)$ is equivalent to a hyperbolic Belyi map $(\mathbb{D}, g,a,b)$. Since $g:\mathbb{D}\rightarrow \mathbb{D}$ is a nonconstant proper holomorphic mapping, $g$ is a finite Blaschke product by a theorem of Fatou \cite[p. 212]{Remmert}. Hence $(g,a,b)$  is a Shabat-Blaschke product. Conversely, if $(X,f,a,b)$ is equivalent to $(\mathbb{D},g,a,b)$, where $(g,a,b)$ is a Shabat-Blaschke product, then $\chi(X)=1$, so $c_1+c_2=n+1$, i.e. its corresponding transitive monodromy representation is a tree.
\end{proof}

Via the transitive monodromy representations, we obtain for each fixed $\lambda\in (0,+\infty)$, a bijective correspondence between the equivalence classes of Belyi maps onto the Riemann sphere and the equivalence classes of hyperbolic Belyi maps of modulus $\lambda$. By Belyi's theorem, in each equivalent class of Belyi map onto the Riemann sphere, there is a representative element $(X,f,a,b,c)$ such that $X$ and $f$ are defined over $\overline{\mathbb{Q}}$. There  is a group action by the absolute Galois group $Gal(\overline{\mathbb{Q}}/\mathbb{Q})$ on the set of equivalence classes of Belyi maps onto the Riemann sphere, which is defined by acting on the algebraic coefficients of $X$ and $f$  \cite[p. 115-117]{Zvonkin}\cite[p. 250]{Girondo}\cite{Grothendieck}.  Thus for each fixed $\lambda$, there is an induced Galois action on the set of equivalence classes of hyperbolic Belyi maps  of modulus $\lambda$. Similarly, there is a Galois action on the set of equivalence classes of Shabat polynomials and we  have the following:

\begin{theorem}
For each fixed $\lambda\in (0,+\infty)$, there is a bijective correspondence between the equivalence classes of Shabat polynomials and the equivalence classes of Shabat-Blaschke products of modulus $\lambda$,  we also have an induced Galois action on the set of equivalence classes of Shabat-Blaschke products of modulus $\lambda$.
\end{theorem}

\begin{proof}
Theorem \ref{trees and Shabat polynomials} says that there is a bijective correspondence between the equivalence classes of tree monodromies and the equivalence classes of Shabat polynomials. Theorem \ref{Shabat-Blaschke products and trees} implies that for each fixed $\lambda\in (0,+\infty)$, there is a bijective correspondence between the equivalence classes of tree monodromies and the equivalence classes of Shabat-Blaschke products of modulus $\lambda$. The first claim follows by compositing the two correspondences.  The Galois action on the set of equivalence classes of Shabat-Blaschke products of modulus $\lambda$ is induced from that on the set of  equivalence classes of Shabat polynomials.
\end{proof}
~\\

\section{Size of hyperbolic dessins d'enfant in $\mathbb{D}$}\label{dessin}

The \textbf{dessin d'enfant} of a Belyi map $(X,f,a,b,c)$ onto the Riemann sphere is defined to be the preimage of a geodesic in $\widehat{\mathbb{C}}$ joining $a$ and $b$ \cite[p. 80]{Zvonkin}\cite{Grothendieck}. We can define the \textbf{hyperbolic dessin d'enfant} of a hyperbolic Belyi map $(X,f,a,b)$ similarly  to be the preimage of the geodesic $l$ joining $a$ and $b$ under the Poincar\'{e} metric. We regard the preimages of $a$ and $b$ as the white and black vertices respectively of the dessin, while the $n$ liftings of the geodesic as the edges of the dessin, where $n$ is the degree of $f$.  If the associated transitive monodromy representation of a (hyperbolic) Belyi map is a tree, then by Lemma \ref{number of cycles}, $|f^{-1}(a)|+|f^{-1}(b)|=c_1+c_2=n+1$, so its associated dessin has $n+1$ vertices and $n$ edges, and such dessin is a connected bipartite tree embedded in $X$, where $X=\widehat{\mathbb{C}}$ if the Belyi map is onto $\widehat{\mathbb{C}}$, and $X=\mathbb{D}$ if the Belyi map is hyperbolic.\\

Let $(B,a,b)$ be a Shabat-Blaschke product and $l$ be the geodesic joining $a$ and $b$ under the Poincar\'{e} metric of $\mathbb{D}$. Then $\mathbb{D}\setminus B^{-1}(l)$ is doubly connected and the modulus of it will be called the {\bf size} of the hyperbolic dessin d'enfant $B^{-1}(l)$.\\

Let $A_r:=\{z: r<|z|<1\}$. We have the following lemma:
\begin{lemma}\label{unbranched covering annulus}
If $f:U\rightarrow A_r$ is a proper unbranched covering of degree $n$, then $U$ is biholomorphic to the doubly-connected domain  $A_{r^{1/n}}$.
\end{lemma}

\begin{proof}
The fundamental group $\pi_1(A_{r})$ is $\mathbb{Z}$. There is a bijective correspondence between the subgroups of $\pi_1(A_{r})$ and the power maps $g_m:A_{r^{1/m}}\rightarrow A_{r}$, $m\geq 1$, defined by $g_m(z)=z^m$. This implies that the power maps are all the proper unbranched coverings of $A_{r}$ up to isomorphism. Since $f$ is of degree $n$, we know that $f$ is isomorphic to the power map $g_n$, so the domain $U$ is biholomorphic to $A_{r^{1/n}}$.
\end{proof}

\begin{theorem}
Suppose $(B,a,b)$ is a Shabat-Blaschke product of degree $n$ and of modulus $\lambda$. Let $l$ be the geodesic between $a$ and $b$ under the Poincar\'{e} metric of $\mathbb{D}$.  Then the modulus of $\mathbb{D}-B^{-1}(l)$ is $\lambda/n$.
\end{theorem}

\begin{proof}
There exists a biholomorphism $\phi$ that maps $\mathbb{D}-l$ onto an annulus $A_r$ for some $r\in (0,1)$. The map $\phi\circ B: \mathbb{D}-B^{-1}(l)\rightarrow A_r$ is a proper unbranched covering. By Lemma \ref{unbranched covering annulus}, we have
$$M(\mathbb{D}-B^{-1}(l))=M(A_{r^{1/n}})=\frac{1}{2\pi}\log(1/r^{1/n})=\frac{1}{n}M(A_r)=\frac{1}{n}M(\mathbb{D}-l)=\lambda/n.$$
\end{proof}

Loosely speaking, the above theorem showed quantitatively how the size of a hyperbolic dessin d'enfant of a Shabat-Blaschke product depends on the degree and the modulus of the Shabat-Blaschke product. Note  however that there is no such concept of the size of a dessin d'enfant of a Shabat polynomial, since the Riemann sphere with a connected tree taken away is a simply-connected domain.\\

\section{Jacobi elliptic functions}\label{Jacobi theta functions and Jacobi elliptic functions}
For any $\tau$ in the open upper-half plane $\mathbb{H}$, and $q=e^{2\pi i \tau}$, the \textbf{Jacobi theta functions} are defined as follows: 
$$\vartheta_1(v,\tau)=\sum_{n=-\infty}^\infty i^{2n-1} q^{(n+1/2)^2} e^{(2n+1)iv},$$
$$\vartheta_2(v,\tau)=\sum_{n=-\infty}^\infty q^{(n+1/2)^2} e^{(2n+1)iv},$$
$$\vartheta_3(v,\tau)=\sum_{n=-\infty}^\infty q^{n^2} e^{2niv},$$
$$\vartheta_0(v,\tau)=\sum_{n=-\infty}^\infty (-1)^n q^{n^2} e^{2niv}.$$
We also define
$$\omega_1(\tau)= \vartheta_3^2(0,\tau),\quad k(\tau)=\frac{\vartheta_2^2(0,\tau)}{\vartheta_3^2(0,\tau)}, \quad \sqrt{k(\tau)}=\frac{\vartheta_2(0,\tau)}{\vartheta_3(0,\tau)}.$$
The theta functions can be expressed in terms of each other, for example:
\begin{equation}\label{half period relation}
\vartheta_0(v,\tau)=\vartheta_3(v+1/2,\tau).
\end{equation}
The modular transformations \cite[p. 475]{Whittaker}\cite[p. 17]{Lawden} for the theta functions $\vartheta_3$ and $\vartheta_2$ are:
\begin{equation}\label{modular 1}
\vartheta_3(v,\tau+1)=\vartheta_3(v,\tau),
\end{equation}
\begin{equation}\label{modular 3}
\vartheta_2(v,\tau+1)=\vartheta_2(v,\tau),
\end{equation}
\begin{equation}\label{modular 2}
\vartheta_3(v,-1/\tau)=(-i\tau/2)^{1/2}e^{\frac{i\tau v^2}{2\pi}}\vartheta_3(\tau v/2,\tau/4),
\end{equation}
and
\begin{equation}\label{modular 4}
\vartheta_2(v,-1/\tau)=(-i\tau/2)^{1/2}e^{\frac{i\tau v^2}{2\pi}}\vartheta_0(\tau v/2,\tau/4).
\end{equation}
We also have 
\begin{equation}\label{modular 5}
\vartheta_3(0,\tau -1/2)=\vartheta_0(0,\tau)
\end{equation}
and 
\begin{equation}\label{quartic relation}
\vartheta_3^4(0,\tau)=\vartheta_2^4(0,\tau)+\vartheta_0^4(0,\tau).
\end{equation}
The \textbf{Jacobi elliptic functions} are defined as follows:
$$sn(u,\tau)=\frac{\vartheta_3(0,\tau)}{\vartheta_2(0,\tau)}\cdot \frac{\vartheta_1(u/\omega_1(\tau),\tau)}{\vartheta_0(u/\omega_1(\tau),\tau)},$$
$$cn(u,\tau)=\frac{\vartheta_0(0,\tau)}{\vartheta_2(0,\tau)}\cdot \frac{\vartheta_2(u/\omega_1(\tau),\tau)}{\vartheta_0(u/\omega_1(\tau),\tau)},$$
$$dn(u,\tau)=\frac{\vartheta_0(0,\tau)}{\vartheta_3(0,\tau)}\cdot \frac{\vartheta_3(u/\omega_1(\tau),\tau)}{\vartheta_0(u/\omega_1(\tau),\tau)},$$
$$cd(u,\tau)=\frac{cn(u,\tau)}{dn(u,\tau)}.$$
It is known that \cite[p. 26]{Lawden}
\begin{equation}\label{cd tends to cosine}
\lim_{\tau\rightarrow +i\infty} cd(u,\tau)=\cos(u).
\end{equation}
~\\

\section{Chebyshev-Blaschke products}
Let $n$ be a positive integer,  $\tau\in \mathbb{R}_{>0}i$, and $\mathbb{D}$ be the open unit disk. The \textbf{Chebyshev-Blaschke product} $f_{n,\tau}:\mathbb{D}\rightarrow \mathbb{D}$ introduced in \cite{WangNg}\cite{Tsang}\cite{Tsang1} is defined by
\begin{equation}\label{Blaschke product: functional definition}
f_{n,\tau}(z)=\sqrt{k(n\tau)}cd(n\omega_1(n\tau)u,n\tau),\end{equation}
where $$z=\sqrt{k(\tau)}cd(\omega_1(\tau)u,\tau).$$
The Chebyshev-Blaschke products $f_{n,\tau}, \tau\in\mathbb{R}_{>0}i$, are hyperbolic analogues of the \textbf{Chebyshev polynomial} $T_n:\mathbb{C}\rightarrow \mathbb{C}$ defined by
$$T_n(z)=cos(nu), \quad z=cos(u).$$
Chebyshev polynomials are examples of Shabat polynomials since they have exactly two critical values in $\mathbb{C}$.\\

If we regard $f_{n,\tau}$ as a rational function on $\widehat{\mathbb{C}}$ and let $\mathcal{T}_{n,\tau}:\widehat{\mathbb{C}}\rightarrow \widehat{\mathbb{C}}$ be defined by
$$\mathcal{T}_{n,\tau}(z)=\frac{f_{n,\tau}(\sqrt{k(\tau)}z)}{\sqrt{k(n\tau)}},$$
then $\mathcal{T}_{n,\tau}$ is referred to as an \textbf{elliptic rational function}. The elliptic rational functions have applications in filter design in engineering \cite[Chapter 12]{Lutovac}. It was shown in \cite{Tsang1} that 
\begin{equation}\label{elliptic rational tends to Chebyshev}
\lim_{\tau\rightarrow +i\infty}\mathcal{T}_{n,\tau}(z)=T_n(z).
\end{equation}

The zeros of this Chebyshev-Blaschke product is computed in \cite{Tsang}\cite{Tsang1}, so it can also be written as a rational function,
\begin{equation}\label{f as a product of Blaschke factor}
f_{n,\tau}(z)=
z^{(1-(-1)^n)/2}\prod_{i=1}^{\lfloor n/2\rfloor}\frac{z^2-b_i}{1-b_iz^2},
\end{equation}
where 
\begin{equation}\label{bi}
b_i=\frac{\vartheta_2^2((2i-1)\pi/2n,\tau)}{\vartheta_3^2((2i-1)\pi/2n,\tau)}, \quad 1\leq i\leq \lfloor n/2\rfloor.
\end{equation}

The Chebyshev-Blaschke product $f_{n,\tau}$ has exactly two critical values  in $\mathbb{D}$, $\sqrt{k(n\tau)}$ and $-\sqrt{k(n\tau)}$ \cite{Tsang1}. Therefore, the Chebyshev-Blaschke products $f_{n,\tau}$, $n\geq 1$, $\tau\in \mathbb{R}_{>0}i$, are examples of Shabat-Blaschke products, whose modulus is $\lambda_{n,\tau}=n\pi\tau/(4i)$ \cite{WangNg}. For all $\tau$, the monodromy of the Chebyshev-Blaschke product $f_{n,\tau}$ is the same as that of the Chebyshev polynomial $T_n$ \cite{WangNg}\cite{Tsang1}. The hyperbolic dessins d'enfant of the Chebyshev-Blaschke products are chains in $\mathbb{D}$.\\

\section{Rings of definition of Chebyshev-Blaschke products}\label{The coefficients and some Landen-type theta identities}

For each $j=1,\dots, \lfloor n/2\rfloor$, let 
\begin{equation}\label{coefficients in terms of theta}
S_{n,j}=\sum_{1\leq i_1<\cdots< i_j\leq \lfloor n/2\rfloor } b_{i_1}\dots b_{i_j},
\end{equation}
where $b_i$ is given by \eqref{bi}.
The $S_{n,j}$ are $\pm$ of the coefficients of $f_{n,\tau}$ when expanded, as
\begin{equation}\label{expanded}
f_{n,\tau}(z)=z^{(1-(-1)^n)/2}\frac{z^{2\lfloor n/2\rfloor}+\sum_{j=1}^{\lfloor n/2\rfloor}(-1)^jS_{n,j}(\tau)z^{2\lfloor n/2\rfloor -2j}}{1+\sum_{j=1}^{\lfloor n/2\rfloor}(-1)^jS_{n,j}(\tau)z^{2j}}.
\end{equation}

Since the $b_i$'s, as functions in $\tau$, are actually defined and meromorphic on the open upper half plane $\mathbb{H}$, so are $S_{n,j}$ and $f_{n,\tau}^{(k)}(0)$.

\begin{lemma}\label{power series coefficients and symmetric functions of zeros}
For each positive integer $n\geq 2$, and $j=1,\dots,\lfloor n/2\rfloor$, $$S_{n,j}\in\mathbb{Q}(f_{n,\tau}(0),f_{n,\tau}^\prime(0),f_{n,\tau}^{\prime\prime}(0),\dots),$$
a subfield of the field $\mathcal{M}(\mathbb{H})$ of meromorphic functions on the upper half plane.
\end{lemma}

\begin{proof}
From \eqref{expanded} we have for each $\tau\in\mathbb{R}_{>0}i$,
$$f_{n,\tau}(z)\left(1+\sum_{j=1}^{\lfloor n/2\rfloor}(-1)^jS_{n,j}z^{2j}\right)=z^{(1-(-1)^n)/2}\left(z^{2\lfloor n/2\rfloor}+\sum_{j=1}^{\lfloor n/2\rfloor}(-1)^jS_{n,j}z^{2\lfloor n/2\rfloor -2j}\right).$$
By expanding $f_{n,\tau}$ in power series at $0$ and comparing coefficients on both sides, and using the identity theorem, the tuple  $(S_{n,1},\dots, S_{n,\lfloor n/2\rfloor})$ satisfies a system ($*$) of countably many linear equations in $\lfloor n/2\rfloor$ $\mathcal{M}(\mathbb{H})$-variables whose coefficients are in
$$\mathbb{Q}(f_{n,\tau}(0),f_{n,\tau}^\prime(0),f_{n,\tau}^{\prime\prime}(0),\dots).$$
On the other hand, suppose $(T_1(\tau),\dots, T_{\lfloor n/2\rfloor}(\tau))\in \mathcal{M}(\mathbb{H})^{\lfloor n/2\rfloor}$ is a solution to ($*$). Since the poles of $T_j$ are discrete, there exists $\tau\in\mathbb{R}_{>0}i$ and a small open ball $\mathcal{N}$ around $\tau$ such that $T_j$ is analytic on $\mathcal{N}$ for all $j$.  Then for each $\tau\in\mathbb{R}_{>0}i\cap\mathcal{N}$, we have
$$f_{n,\tau}(z)\left(1+\sum_{j=1}^{\lfloor n/2\rfloor}(-1)^jT_{j}z^{2j}\right)=z^{(1-(-1)^n)/2}\left(z^{2\lfloor n/2\rfloor}+\sum_{j=1}^{\lfloor n/2\rfloor}(-1)^jT_{j}z^{2\lfloor n/2\rfloor -2j}\right)$$
on $\mathbb{D}$, and hence on $\mathbb{P}^1$ by the identity theorem. Therefore,
\begin{eqnarray}
&&z^{(1-(-1)^n)/2}\frac{z^{2\lfloor n/2\rfloor}+\sum_{j=1}^{\lfloor n/2\rfloor}(-1)^jS_{n,j}z^{2\lfloor n/2\rfloor -2j}}{1+\sum_{j=1}^{\lfloor n/2\rfloor}(-1)^jS_{n,j}z^{2j}}\nonumber
\\&=& z^{(1-(-1)^n)/2}\frac{z^{2\lfloor n/2\rfloor}+\sum_{j=1}^{\lfloor n/2\rfloor}(-1)^jT_{j}z^{2\lfloor n/2\rfloor -2j}}{1+\sum_{j=1}^{\lfloor n/2\rfloor}(-1)^jT_{j}z^{2j}}\nonumber
\end{eqnarray}
 on $\mathbb{P}^1$. Since the numerators of both sides are monic, we have $S_{n,j}(\tau)=T_j(\tau)$ for all $j=1,\dots, \lfloor n/2\rfloor$. Since this holds for all $\tau\in\mathbb{R}_{>0}i\cap\mathcal{N}$, by the identity theorem, $S_{n,j}=T_j$ on the upper half plane, for all $j$.  This shows that   $(S_{n,1},\dots, S_{n,\lfloor n/2\rfloor})$ is the unique solution to ($*$). By Gaussian elimination,  the infinite system ($*$) is equivalent to a system ($**$) of  at most $\lfloor n/2\rfloor$ linear equations in  $\lfloor n/2\rfloor$ $\mathcal{M}(\mathbb{H})$-variables, whose coefficients are again in
$$\mathbb{Q}(f_{n,\tau}(0),f_{n,\tau}^\prime(0),f_{n,\tau}^{\prime\prime}(0),\dots).$$ Now $(S_{n,1},\dots,  S_{n,\lfloor n/2\rfloor})$ is the unique solution to ($**$), so the system ($**$) should have exactly $\lfloor n/2\rfloor$ linear equations, and by Cramer's rule,
$$S_{n,j}\in\mathbb{Q}(f_{n,\tau}(0),f_{n,\tau}^\prime(0),f_{n,\tau}^{\prime\prime}(0),\dots)$$
for all $j=1,\dots, \lfloor n/2\rfloor$.
\end{proof}

From \eqref{f as a product of Blaschke factor}, we know that the Chebyshev-Blaschke product $f_{n,\tau}$ is even if $n$ is even and it is odd if $n$ is odd. 
By \eqref{Blaschke product: functional definition} and the properties on the Jacobi elliptic functions \cite[p. 499-500]{Whittaker},  we have
$$
f_{n,\tau}(0)=
\begin{dcases}
(-1)^{n/2}\sqrt{k(n\tau)} \quad &\text{if $n$ is even};
\\0  \quad &\text{if $n$ is odd}.
\end{dcases}
$$
By differentiating \eqref{Blaschke product: functional definition} and putting $z=0$, we have
$$
f_{n,\tau}^\prime(0)=
\begin{dcases}
0\quad &\text{if $n$ is even};
\\(-1)^{(n-1)/2}\cdot\frac{n\omega_1(n\tau)\sqrt{k(n\tau)}}{\omega_1(\tau)\sqrt{k(\tau)}}  \quad &\text{if $n$ is odd}.
\end{dcases}
$$
It was proved in \cite{Tsang} that for each $n\geq 1$ and $\tau\in \mathbb{R}_{>0}i$, the Chebyshev-Blaschke product $f_{n,\tau}$ satisfies the following nonlinear differential equation:
\begin{eqnarray}
&\omega_1^2(\tau)(k(\tau)-z^2)(1-k(\tau)z^2)\frac{d^2w}{dz^2}+\omega_1^2(\tau)[2k(\tau)z^3-(1+k^2(\tau))z]\frac{dw}{dz} \nonumber
\\&+n^2\omega_1^2(n\tau)[(1+k^2(n\tau))w-2k(n\tau)w^3]=0. \label{nonlinear differential equation}
\end{eqnarray}
Hence we have
$$
f_{n,\tau}^{\prime\prime}(0)=
\begin{dcases}
(-1)^{n/2}\cdot\frac{n^2\omega_1^2(n\tau)\sqrt{k(n\tau)}}{\omega_1^2(\tau)k(\tau)}(k^2(n\tau)-1)\quad &\text{if $n$ is even};
\\ 0\quad &\text{if $n$ is odd}.
\end{dcases}
$$
Let $$A(z):=\omega_1^2(\tau)(k(\tau)-z^2)(1-k(\tau)z^2)$$ and $$B(z):=\omega_1^2(\tau)[2k(\tau)z^3-(1+k^2(\tau))z].$$
Denote the binomial coefficients by $C^\alpha_\beta$. By applying Leibniz rule twice, for $i\geq 2$,
\begin{eqnarray}
(w^3)^{(i)}&=&3(w^2w^{(1)})^{(i-1)}\nonumber
\\&=&3\sum_{j=0}^{i-1}C_j^{i-1}(w^2)^{(j)}w^{(i-j)}\nonumber
\\&=&3w^2w^{(i)}+3\sum_{j=1}^{i-1}C_j^{i-1}(2ww^{(1)})^{(j-1)}w^{(i-j)}\nonumber
\\&=&3w^2w^{(i)}+6\sum_{j=1}^{i-1}\sum_{k=0}^{j-1}C_j^{i-1}C_k^{j-1}w^{(k)}w^{(j-k)}w^{(i-j)},\nonumber
\end{eqnarray}
so by taking the $i$-th derivative of \eqref{nonlinear differential equation},  we have for $i\geq 2$,
\begin{eqnarray}
&&Aw^{(i+2)}+\sum_{j=1}^i(C_j^i A^{(j)}+C^i_{j-1} B^{(j-1)}) w^{(i-j+2)}+B^{(i)} w^\prime+n^2\omega_1^2(n\tau)(1+k^2(n\tau))w^{(i)} \nonumber
\\&&-6n^2\omega_1^2(n\tau)k(n\tau)\left[w^2w^{(i)}+2\sum_{j=1}^{i-1}\sum_{k=0}^{j-1}C_j^{i-1}C_k^{j-1}w^{(k)}w^{(j-k)}w^{(i-j)}\right]=0. \label{i-th derivative}
\end{eqnarray}
For $j\geq 5$, the $j$-th derivatives of the polynomials $A$ and $B$ vanish, so for $i\geq 4$,
\begin{eqnarray}
&&Aw^{(i+2)}+\sum_{j=1}^4 (C_j^i A^{(j)}+C^i_{j-1} B^{(j-1)}) w^{(i-j+2)}+n^2\omega_1^2(n\tau)(1+k^2(n\tau))w^{(i)}\nonumber
\\&&-6n^2\omega_1^2(n\tau)k(n\tau)\left[w^2w^{(i)}+2\sum_{j=1}^{i-1}\sum_{k=0}^{j-1}C_j^{i-1}C_k^{j-1}w^{(k)}w^{(j-k)}w^{(i-j)}\right]=0.\nonumber
\end{eqnarray}
Putting $w=f_{n,\tau}$, $z=0$, we get for $i\geq 4$,
\begin{eqnarray}
&&\omega_1^2(\tau)k(\tau)f^{(i+2)}_{n,\tau}(0)-i^2\omega_1^2(\tau)(1+k^2(\tau))f_{n,\tau}^{(i)}(0)  \nonumber
\\&&+i(i-1)^2(i-2)\omega_1^2(\tau)k(\tau)f_{n,\tau}^{(i-2)}(0)
+n^2\omega_1^2(n\tau)(1+k^2(n\tau))f_{n,\tau}^{(i)}(0) \nonumber
\\&&-6n^2\omega_1^2(n\tau)k(n\tau)\left[f_{n,\tau}(0)^2 f_{n,\tau}^{(i)}(0)
+2\sum_{j=1}^{i-1}\sum_{k=0}^{j-1}C_j^{i-1}C_k^{j-1}f_{n,\tau}^{(k)}(0)f_{n,\tau}^{(j-k)}(0)f_{n,\tau}^{(i-j)}(0)
\right]
=0. \nonumber
\\&&\label{recurrence}
\end{eqnarray}
If $n$ and $i$ are of the same parity and $i\geq 4$, then by \eqref{recurrence},
\begin{eqnarray}
&&f^{(i+2)}_{n,\tau}(0)+\left[\frac{n^2\omega_1^2(n\tau)[1+(3(-1)^{n-1}-2)k^2(n\tau)]}{\omega_1^2(\tau)k(\tau)}-i^2\left(\frac{1}{k(\tau)}+k(\tau)\right)\right]f_{n,\tau}^{(i)}(0)  \nonumber
\\&&+i(i-1)^2(i-2)f_{n,\tau}^{(i-2)}(0) \nonumber
\\&&-\frac{12n^2\omega_1^2(n\tau)k(n\tau)}{\omega_1^2(\tau)k(\tau)}\sum_{j=1}^{i-1}\sum_{k=0}^{j-1}C_j^{i-1}C_k^{j-1}f_{n,\tau}^{(k)}(0)f_{n,\tau}^{(j-k)}(0)f_{n,\tau}^{(i-j)}(0)
=0. \label{recurrence1}
\end{eqnarray}
Differentiating \eqref{nonlinear differential equation} once, and putting $w=f_{n,\tau}$ and $z=0$, we have
\begin{eqnarray}
&&\omega_1^2(\tau)k(\tau)f_{n,\tau}^{(3)}(0)+[n^2\omega_1^2(n\tau)(1+k^2(n\tau))-\omega_1^2(\tau)(1+k^2(\tau))] f_{n,\tau}^\prime(0)\nonumber
\\&&-6n^2\omega_1^2(n\tau)k(n\tau)f_{n,\tau}^2(0)f_{n,\tau}^\prime(0)=0.\nonumber
\end{eqnarray}
Then if $n$ is odd,
$$
f_{n,\tau}^{(3)}(0)=(-1)^{(n+1)/2}\cdot\frac{n\omega_1(n\tau)\sqrt{k(n\tau)}}{\omega_1(\tau)\sqrt{k(\tau)}^3}
\left[\frac{n^2\omega_1^2(n\tau)}{\omega_1^2(\tau)}(1+k^2(n\tau))-(1+k^2(\tau))\right].
$$
If $n$ is even, $f_{n,\tau}^{(3)}(0)=0$.
\\By \eqref{i-th derivative}, if $n$ is even,
$$f_{n,\tau}^{(4)}(0)=(-1)^{n/2}\cdot\frac{n^2\omega_1^2(n\tau)\sqrt{k(n\tau)}}{\omega_1^2(\tau)k^2(\tau)}(1-k^2(n\tau))\left[\frac{n^2\omega_1^2(n\tau)}{\omega_1^2(\tau)}(1-5k^2(n\tau))-4(1+k^2(\tau))\right].$$
If $n$ is odd, then $f_{n,\tau}^{(4)}(0)=0$. 
\\By \eqref{i-th derivative} again, if $n$ is odd,
\begin{eqnarray}
f_{n,\tau}^{(5)}(0)
&=&(-1)^{(n-1)/2}\cdot\frac{n\omega_1(n\tau)\sqrt{k(n\tau)}}{\omega_1(\tau)\sqrt{k(\tau)}^5}\bigg[\frac{n^4\omega_1^4(n\tau)}{\omega_1^4(\tau)}(k^4(n\tau)+14k^2(n\tau)+1)\nonumber
\\&&\quad-10\frac{n^2\omega_1^2(n\tau)}{\omega_1^2(\tau)}(1+k^2(\tau))(1+k^2(n\tau))+3(3k^4(\tau)+2k^2(\tau)+3)\bigg].\nonumber
\end{eqnarray}
If $n$ is even, then $f_{n,\tau}^{(5)}(0)=0$. 
By the  expressions for $f_{n,\tau}(0),f_{n,\tau}^\prime(0),\cdots, f_{n,\tau}^{(5)}(0)$, together with  \eqref{recurrence1}  and induction, we know that for each $n$ and $i$, as a function in $\tau$ on the upper half plane,
$$f_{n,\tau}^{(i)}(0)\in\mathbb{Q}\left(\sqrt{k(\tau)},\sqrt{k(n\tau)}, \frac{\omega_1(n\tau)}{\omega_1(\tau)}\right).$$
Hence by Lemma \ref{power series coefficients and symmetric functions of zeros}, we have the following:
\begin{theorem}\label{CB-product coefficients in field of modulus}
For each $n\geq 1$, the coefficients of the Chebyshev-Blaschke product $f_{n,\tau}$ are   in the field
$$\mathbb{Q}\left(\sqrt{k},\sqrt{k\circ s_n}, \frac{\omega_1\circ s_n}{\omega_1}\right),$$
where $s_n(\tau)=n\tau$. By multiplying both the numerator and denominator of $f_{n,\tau}$ by the product of the denominators of the $S_{n,j}$, $j=1,\dots, \lfloor n/2\rfloor$, we know that $f_{n,\tau}$ is defined over the ring
$$\mathbb{Z}\bigg[\sqrt{k},\sqrt{k\circ s_n}, \frac{\omega_1\circ s_n}{\omega_1}\bigg].$$
Moreover,$$\frac{\vartheta_2((2i-1)\pi/2n,\tau)}{\vartheta_3((2i-1)\pi/2n,\tau)},\qquad i=1,\dots, \lfloor n/2\rfloor,$$
are algebraic over this field since they are zeros of $f_{n,\tau}$.
\end{theorem}
This theorem is analogous to the fact that the coefficients of the Chebyshev polynomials are in $\mathbb{Z}$, and also the fact that   
for each $n\geq 2$ and $i=1,\dots, \lfloor n/2\rfloor$,
$$\cos\left(\frac{(2i-1)\pi}{2n}\right)$$
is an algebraic number.
Note also that when $\tau\rightarrow +i\infty$, 
$$\sqrt{k(\tau)}\rightarrow 0 \quad \text{and} \quad \frac{\omega_1(n\tau)}{\omega_1(\tau)}\rightarrow 1,$$
so the ring in the theorem degenerates to $\mathbb{Z}$ when $\tau\rightarrow +i\infty$.

\begin{remark}
Another result similar to Theorem \ref{CB-product coefficients in field of modulus} is Theorem 4.1 in the paper \cite{Ismail} of Ismail and Zhang. It was proved that the coefficients of Ramanujan entire functions are lying in a polynomial ring over $\mathbb{C}(q)$ generated by expressions in terms of $q^{1/4}, \vartheta_2(0,\tau)$ and $\vartheta_3(0,\tau)$.
\end{remark}

Next, we will prove that the coefficients of the Chebyshev-Blaschke products are in the algebraic closure $\overline{\mathbb{Q}(j)}$, where $j$ is the $j$-invariant.

\begin{lemma}
For any positive integer $n$, $\vartheta_2^4(0,n\tau)$ and $\vartheta_3^4(0,n\tau)$ are modular forms of weight $2$ with respect to the Hecke congruence subgroup $\Gamma_0(4n)$.
\end{lemma}

\begin{proof}
By \eqref{half period relation}, \eqref{modular 4} and \cite[p. 338]{Freitag}, both $\vartheta_3(0,\tau)$ and $\vartheta_2(0,\tau)$ are holomorphic on the open upper half plane $\mathbb{H}$. By \eqref{modular 1} and \eqref{modular 2}, we have $$\vartheta_3(0,\tau+1)=\vartheta_3(0,\tau) \quad\text{ and }\quad \vartheta_3\left(0,\frac{\tau}{4\tau+1}\right)=(4\tau+1)^{1/2}\vartheta_3(0,\tau).$$
Since $\Gamma_0(4)$ is generated by 
$$\pm\left(
\begin{array}{cc}
1 & 1
\\0 & 1
\end{array}
\right)
\quad \text{and}\quad
\pm\left(
\begin{array}{cc}
1 & 0
\\4 & 1
\end{array}
\right),
$$
see \cite[p. 21]{Diamond}, we have that $\vartheta_3^4(0,\tau)$ is weight 2 invariant under $\Gamma_0(4)$. Similarly, by \eqref{modular 3} and \eqref{modular 4}, $\vartheta_2^4(0,\tau)$ is weight $2$ invariant under  $\Gamma_0(4)$. It is known that the $n$th-coefficient $d_n$ of the Fourier series of $\vartheta_3^4(0,\tau)$ is the number of ways to express $n$ as an ordered sum of squares of four integers, so $|d_n|\leq (2n+1)^4\leq 3^4n^4$ for $n>0$, so the $n$-th coefficient $e_n$ of the Fourier series in $q^{1/4}$ satisfies $|e_n|\leq (3/4)^4n^4$ for $n>0$. By Proposition 1.2.4 in \cite[p. 17]{Diamond}, $\vartheta_3^4(0,\tau)$ is a weight 2 modular form with respect to $\Gamma_0(4)$. Let $\widetilde{q}=q^{1/4}$. Then
$$\vartheta_2^4(0,\tau)=\left(\sum_{n=-\infty}^\infty \widetilde{q}^{(2n+1)^2}\right)^4$$
whose $n$-th coefficient $c_n$ count the number of ways to express $n$ as an ordered sum of squares of four odd integers, so $c_n$ satisfies a similar bound for $n>0$, and hence $\vartheta_2^4(0,\tau)$ is a weight 2 modular form with respect to $\Gamma_0(4)$. By Lemma \ref{isogeny modular forms}, $\vartheta_2^4(0,n\tau)$ and $\vartheta_3^4(0,n\tau)$ are weight 2 modular forms with respect to $\Gamma_0(4n)$.
\end{proof}

\begin{corollary}\label{elliptic modulus are modular}
The functions $k^2(\tau), k^2(n\tau)$, and $\frac{\omega_1^2(n\tau)}{\omega_1^2(\tau)}$ are modular functions with respect to $\Gamma_0(4n)$.
\end{corollary}

\begin{theorem}
For each positive integer $n$, the coefficients of the  Chebyshev-Blaschke product $f_{n,\tau}$  are in the algebraic closure $\overline{\mathbb{Q}(j)}$, where $j$ is the $j$-invariant.
\end{theorem}

\begin{proof}
Since the coefficients of the Fourier expansions at $\infty$ of $\vartheta_3^4(0,n\tau)$ and $\vartheta_2^4(0,n\tau)$ are rational, those of $k^2(\tau)$, $k^2(n\tau)$ and $\frac{\omega_1^2(n\tau)}{\omega_1^2(\tau)}$ are also rational. Hence by Corollary \ref{elliptic modulus are modular} and Lemma \ref{rational Fourier coefficients}, 
$$k^2(\tau), k^2(n\tau), \frac{\omega_1^2(n\tau)}{\omega_1^2(\tau)}\in \mathbb{Q}(j,j_{4n}),$$
where $j_{4n}(\tau)=j(4n\tau)$.
By Lemma \ref{modular equations}, there exists nonconstant polynomial $\Phi_{4n}\in \mathbb{Q}[X,Y]$ such that $$\Phi_{4n}(j,j_{4n})=0.$$ This shows that $j_{4n}\in \overline{\mathbb{Q}(j)}$, so $\mathbb{Q}(j,j_{4n})\subseteq \overline{\mathbb{Q}(j)}$. Then 
$$\sqrt{k(\tau)}, \sqrt{k(n\tau)}, \frac{\omega_1(n\tau)}{\omega_1(\tau)}\in \overline{\mathbb{Q}(j)}.$$
 Hence by Theorem \ref{CB-product coefficients in field of modulus}, the coefficients of $f_{n,\tau}$ are in $\overline{\mathbb{Q}(j)}$.
\end{proof}

\begin{theorem}
For each positive integer $n$, the coefficients of the Chebyshev-Blaschke product $f_{n,\tau}$ are in $Frac(\mathbb{Z}[[q^{1/4}]])$, the field of fraction of the ring of power series in $q^{1/4}$ over $\mathbb{Z}$. By multiplying both the numerator and the denominator of $f_{n,\tau}$ by a suitable element in $\mathbb{Z}[[q^{1/4}]]$, the denominators of the coefficients are cleared out and hence $f_{n,\tau}$ is defined over $\mathbb{Z}[[q^{1/4}]]$.
\end{theorem}

\begin{proof}
By abuse of notations, we use $\sqrt{k(q)}$ and $\omega_1(q)$ to denote the $q$-expansions of $\sqrt{k(\tau)}$ and $\omega_1(\tau)$ respectively. Since 
$$
\begin{array}{ll}
\mathbb{Q}\left(\sqrt{k(q)}, \sqrt{k(q^n)}, \frac{\omega_1(q^n)}{\omega_1(q)}\right)&\subseteq \mathbb{Q}(\vartheta_2(0,q),\vartheta_3(0,q),\vartheta_2(0,q^n),\vartheta_3(0,q^n))\\
&\subseteq Frac(\mathbb{Z}[[q^{1/4}]]),
\end{array}
$$
by Theorem \ref{CB-product coefficients in field of modulus}, the coefficients of $f_{n,\tau}$ are in $Frac(\mathbb{Z}[[q^{1/4}]])$.
\end{proof}

\begin{remark}
It would be interesting to find more examples of other family of Shabat-Blaschke products that are defined over a finite extension of $Frac(\mathbb{Z}[[q^{1/4}]])$, or defined over a finite extension of $\mathbb{Q}\left(\sqrt{k}, \sqrt{k\circ s_n}, \frac{\omega_1\circ s_n}{\omega_1}\right)$, or defined over $\overline{\mathbb{Q}(j)}$. One would also like to see if there is a deformation of the Belyi's theorem formulated using the above fields.\\
\end{remark}

\section{Landen-type identities for theta functions}

In this section, we will obtain some Landen-type identities for theta functions, which will degenerate to some trigonometric identities.\\

From \eqref{coefficients in terms of theta}, we know that for each positive integer $n\geq 2$, and $j=1,\dots, \lfloor n/2\rfloor$,  $S_{n,j}$ can be expressed in terms of $$\vartheta_2((2i-1)\pi/2n,\tau) \quad \text{ and } \quad \vartheta_3((2i-1)\pi/2n,\tau),$$ where $1\leq i\leq \lfloor n/2\rfloor$. On the other hand, we know from Theorem \ref{CB-product coefficients in field of modulus} that $S_{n,j}$ can be expressed in terms of 
$$\vartheta_2(0,\tau), \vartheta_3(0,\tau),\vartheta_2(0,n\tau), \vartheta_3(0,n\tau).$$
Therefore, for each positive integer $n\geq 2$, and each $j=1,\dots, \lfloor n/2\rfloor$, we have a theta identity relating those theta functions. For example, when $n$ is even and $j=\lfloor n/2\rfloor$, we have 
$$\prod_{1\leq i\leq \lfloor n/2\rfloor} \frac{\vartheta_2^2((2i-1)\pi/2n,\tau)}{\vartheta_3^2((2i-1)\pi/2n,\tau)}=\frac{\vartheta_2(0,n\tau)}{\vartheta_3(0,n\tau)}$$
which coincides with the Landen transformation of even order $n$ evaluated at $z=0$ \cite[p. 23, 253-254, 259]{Lawden}.  When $n$ is odd and $j=\lfloor n/2\rfloor$, we have
$$\prod_{1\leq i\leq \lfloor n/2\rfloor} \frac{\vartheta_2^2((2i-1)\pi/2n,\tau)}{\vartheta_3^2((2i-1)\pi/2n,\tau)}=\frac{n\vartheta_2(0,n\tau)\vartheta_3(0,n\tau)}{\vartheta_2(0,\tau)\vartheta_3(0,\tau)}$$
which coincides with the Landen transformation of odd order $n$ evaluated at $z=0$ \cite[p. 253-256]{Lawden}. However, we also get other theta identities in which the left hand sides are other symmetric polynomials in $$ \frac{\vartheta_2^2((2i-1)\pi/2n,\tau)}{\vartheta_3^2((2i-1)\pi/2n,\tau)}, \quad i=1,\dots, \lfloor n/2\rfloor.$$
We display some examples of theta identities when $n$ is small below. When $n=2$, we only have one identity
$$\frac{\vartheta_2^2(\pi/4,\tau)}{\vartheta_3^2(\pi/4,\tau)}=\frac{\vartheta_2(0,2\tau)}{\vartheta_3(0,2\tau)}.$$
By \eqref{cd tends to cosine}, we know that 
$$\lim_{\tau\rightarrow +i\infty}\frac{1}{k(\tau)}\frac{\vartheta_2^2(\pi/4,\tau)}{\vartheta_3^2(\pi/4,\tau)}=\cos^2\left(\frac{\pi}{4}\right).$$
By considering the Fourier expansions, we have
$$\lim_{\tau\rightarrow +i\infty}\frac{1}{k(\tau)}\frac{\vartheta_2(0,2\tau)}{\vartheta_3(0,2\tau)}=\frac{1}{2}.$$
By the above identity and limits, we get
$$\cos^2\left(\frac{\pi}{4}\right)=\frac{1}{2}.$$
When $n=3$, we again only have one identity
$$\frac{\vartheta_2^2(\pi/6,\tau)}{\vartheta_3^2(\pi/6,\tau)}=\frac{3\vartheta_2(0,3\tau)\vartheta_3(0,3\tau)}{\vartheta_2(0,\tau)\vartheta_3(0,\tau)}.$$
By multiplying $1/k(\tau)$ on both sides and taking limits, we get the trigonometric identity
$$\cos^2\left(\frac{\pi}{6}\right)=\frac{3}{4}.$$
When $n=4$, we have two theta identities,
$$\frac{\vartheta_2^2(\pi/8,\tau)}{\vartheta_3^2(\pi/8,\tau)}+\frac{\vartheta_2^2(3\pi/8,\tau)}{\vartheta_3^2(3\pi/8,\tau)}=\frac{8\vartheta_2(0,4\tau)}{\vartheta_2^2(0,\tau)\vartheta_3^2(0,\tau)}\cdot\frac{\vartheta_3^4(0,4\tau)-\vartheta_2^4(0,4\tau)}{\vartheta_3(0,4\tau)-\vartheta_2(0,4\tau)}$$
and 
$$\frac{\vartheta_2^2(\pi/8,\tau)}{\vartheta_3^2(\pi/8,\tau)}\frac{\vartheta_2^2(3\pi/8,\tau)}{\vartheta_3^2(3\pi/8,\tau)}=\frac{\vartheta_2(0,4\tau)}{\vartheta_3(0,4\tau)}.$$
By multiplying $1/k(\tau)$ on both sides of the first identity and taking limits, and by multiplying $1/k^2(\tau)$ on both sides of the second identity and taking limits, we get the trigonometric identities
$$\cos^2\left(\frac{\pi}{8}\right)+\cos^2\left(\frac{3\pi}{8}\right)=1$$
and $$\cos^2\left(\frac{\pi}{8}\right)\cos^2\left(\frac{3\pi}{8}\right)=\frac{1}{8}.$$
When $n=5$, we have two theta identities,
\begin{eqnarray}
&&\frac{\vartheta_2^2(\pi/10,\tau)}{\vartheta_3^2(\pi/10,\tau)}+\frac{\vartheta_2^2(3\pi/10,\tau)}{\vartheta_3^2(3\pi/10,\tau)}\nonumber
\\&=&\frac{5\vartheta_2(0,5\tau)\vartheta_3(0,5\tau)}{6\vartheta_2^2(0,\tau)\vartheta_3^2(0,\tau)}\cdot\frac{\vartheta_3^4(0,\tau)+\vartheta_2^4(0,\tau)-25(\vartheta_3^4(0,5\tau)+\vartheta_2^4(0,5\tau))}{5\vartheta_2(0,5\tau)\vartheta_3(0,5\tau)-\vartheta_2(0,\tau)\vartheta_3(0,\tau)}\nonumber
\end{eqnarray}
and 
$$\frac{\vartheta_2^2(\pi/10,\tau)}{\vartheta_3^2(\pi/10,\tau)}\frac{\vartheta_2^2(3\pi/10,\tau)}{\vartheta_3^2(3\pi/10,\tau)}=\frac{5\vartheta_2(0,5\tau)\vartheta_3(0,5\tau)}{\vartheta_2(0,\tau)\vartheta_3(0,\tau)}.$$
By multiplying $1/k(\tau)$ on both sides of the first identity and taking limits, and by multiplying $1/k^2(\tau)$ on both sides of the second identity and taking limits, we get the trigonometric identities
$$\cos^2\left(\frac{\pi}{10}\right)+\cos^2\left(\frac{3\pi}{10}\right)=\frac{5}{4}$$
and $$\cos^2\left(\frac{\pi}{10}\right)\cos^2\left(\frac{3\pi}{10}\right)=\frac{5}{16}.$$
When $n=6$, we have three theta identities,
\begin{eqnarray}
&&\frac{\vartheta_2^2(\pi/12,\tau)}{\vartheta_3^2(\pi/12,\tau)}+\frac{\vartheta_2^2(3\pi/12,\tau)}{\vartheta_3^2(3\pi/12,\tau)}+\frac{\vartheta_2^2(5\pi/12,\tau)}{\vartheta_3^2(5\pi/12,\tau)}\nonumber
\\&=&\frac{6\vartheta_2(0,6\tau)(\vartheta_2^2(0,6\tau)+\vartheta_3^2(0,6\tau))}{\vartheta_2^2(0,\tau)\vartheta_3^2(0,\tau)}\nonumber
\\&&\cdot\frac{3\vartheta_2^2(0,\tau)\vartheta_3^2(0,\tau)\vartheta_2(0,6\tau)-\vartheta_3(0,6\tau)[\vartheta_2^4(0,\tau)+\vartheta_3^4(0,\tau)+45\vartheta_2^4(0,6\tau)-9\vartheta_3^4(0,6\tau)]}{\vartheta_2^2(0,\tau)\vartheta_3^2(0,\tau)-18\vartheta_2(0,6\tau)\vartheta_3(0,6\tau)(\vartheta_2^2(0,6\tau)+\vartheta_3^2(0,6\tau))},\nonumber
\end{eqnarray}
\begin{eqnarray}
&&\frac{\vartheta_2^2(\pi/12,\tau)}{\vartheta_3^2(\pi/12,\tau)}\frac{\vartheta_2^2(3\pi/12,\tau)}{\vartheta_3^2(3\pi/12,\tau)}+\frac{\vartheta_2^2(\pi/12,\tau)}{\vartheta_3^2(\pi/12,\tau)}\frac{\vartheta_2^2(5\pi/12,\tau)}{\vartheta_3^2(5\pi/12,\tau)}+\frac{\vartheta_2^2(3\pi/12,\tau)}{\vartheta_3^2(3\pi/12,\tau)}\frac{\vartheta_2^2(5\pi/12,\tau)}{\vartheta_3^2(5\pi/12,\tau)}\nonumber
\\&=&\frac{6\vartheta_2(0,6\tau)(\vartheta_2^2(0,6\tau)+\vartheta_3^2(0,6\tau))}{\vartheta_2^2(0,\tau)\vartheta_3^2(0,\tau)}\nonumber
\\&&\cdot\frac{3\vartheta_2^2(0,\tau)\vartheta_3^2(0,\tau)\vartheta_3(0,6\tau)-\vartheta_2(0,6\tau)[\vartheta_2^4(0,\tau)+\vartheta_3^4(0,\tau)-9\vartheta_2^4(0,6\tau)+45\vartheta_3^4(0,6\tau)]}{\vartheta_2^2(0,\tau)\vartheta_3^2(0,\tau)-18\vartheta_2(0,6\tau)\vartheta_3(0,6\tau)(\vartheta_2^2(0,6\tau)+\vartheta_3^2(0,6\tau))},\nonumber
\end{eqnarray}
and
$$\frac{\vartheta_2^2(\pi/12,\tau)}{\vartheta_3^2(\pi/12,\tau)}\frac{\vartheta_2^2(3\pi/12,\tau)}{\vartheta_3^2(3\pi/12,\tau)}\frac{\vartheta_2^2(5\pi/12,\tau)}{\vartheta_3^2(5\pi/12,\tau)}=\frac{\vartheta_2(0,6\tau)}{\vartheta_3(0,6\tau)}.$$
By multiplying $1/k(\tau)$ on both sides of the first identity and taking limits, by multiplying $1/k^2(\tau)$ on both sides of the second identity and taking limits,  by multiplying $1/k^3(\tau)$ on both sides of the third identity and taking limits, we get the trigonometric identities
$$\cos^2\left(\frac{\pi}{12}\right)+\cos^2\left(\frac{3\pi}{12}\right)+\cos^2\left(\frac{5\pi}{12}\right)=\frac{3}{2},$$
$$\cos^2\left(\frac{\pi}{12}\right)\cos^2\left(\frac{3\pi}{12}\right)+\cos^2\left(\frac{\pi}{12}\right)\cos^2\left(\frac{5\pi}{12}\right)+\cos^2\left(\frac{3\pi}{12}\right)\cos^2\left(\frac{5\pi}{12}\right)=\frac{9}{16},$$
and
$$\cos^2\left(\frac{\pi}{12}\right)\cos^2\left(\frac{3\pi}{12}\right)\cos^2\left(\frac{5\pi}{12}\right)=\frac{1}{32}.$$
~\\

\section{Conjecture}\label{epilogue}

Finally, we list some further problems that one may try to study.
Suppose $\rho$ is a tree monodromy representation that gives rise to Shabat-Blaschke products of degree $n$ with exactly two critical values in $\mathbb{D}$. For each fixed $\tau\in\mathbb{R}_{>0}i$, motivated by the critical values of the Chebyshev-Blaschke products, let $B_{\tau}:\mathbb{D}\rightarrow\mathbb{D}$ be the Shabat-Blaschke product associated to $\rho$ whose critical values are $-\sqrt{k(n\tau)}$ and $\sqrt{k(n\tau)}$.  Let $s_n(\tau)=n\tau$ for all $\tau\in\mathbb{H}$.
We have the following conjecture:
\begin{conjecture}
There exists $a_1,\dots, a_n,b_1,\dots, b_m$ in 
$$\overline{\mathbb{Q}\left(\sqrt{k},\sqrt{k\circ s_n}, \frac{\omega_1\circ s_n}{\omega_1}\right)}\subseteq \mathcal{M}(\mathbb{H})$$
(where $\mathcal{M}(\mathbb{H})$ is the field of meromorphic functions on the upper half-plane $\mathbb{H}$) such that 
\begin{itemize}
\item for any $\tau\in \mathbb{R}_{>0}i$, $a_0(\tau),\dots, a_n(\tau), b_0(\tau),\dots, b_m(\tau)\in \mathbb{C}$.
\item for any $\tau\in\mathbb{R}_{>0}i$ and $z\in \mathbb{D}$, $$B_{\tau}(z)=\frac{a_n(\tau)z^n+a_{n-1}(\tau)z^{n-1}+\cdots +a_0(\tau)}{b_m(\tau)z^m+b_{m-1}(\tau)z^{m-1}+\cdots +b_0(\tau)}.$$ 
\end{itemize}
\end{conjecture}
One can also have weaker conjectures by replacing the algebraic closure in the above conjecture by $\overline{\mathbb{Q}(j)}$ or $\overline{Frac(\mathbb{Z}[[q^{1/4}]])}$. \\

In a paper \cite{Maskit} by Maskit, there is a way to embed a topologically finite Riemann surface into a compact Riemann surface. One may try to use such embedding to formulate a deformation of Belyi's theorem.\\

Another version of the Belyi's theorem says that a compact Riemann surface $X$ admits a Belyi map if and only if $X$ can be uniformized by a finite index subgroup of a Fuchsian triangle group \cite[p. 71]{Jones}.  One may try to formulate a deformation of this version of Belyi's theorem.


\begin{bibdiv}
\begin{biblist}

\bib{BM07}{article} {author={A. F. Beardon}, author={D. Minda}, title={The hyperbolic metric and geometric function theory} journal={Proc. of the International Workshop on Quasiconformal Mappings and their Applications, Narosa, New Delhi}, date={2007}, pages={9--56}}

\bib{Belyi}{article} {author={G. V. Belyi}, title={On Galois extensions of a maximal cyclotomic field},  journal={Izv. Akad. Nauk SSSR Ser. Mat.}, volume={43(2)}, date={1979},  pages={267--276}}

\bib{Conway}{book} {author={J. B. Conway}, title={Functions of one complex variable II}, publisher={Springer-Verlag}, date={1995}}

\bib{Cox}{book} {author={D. A. Cox}, title={Primes of the form $x^2+ny^2$: Fermat, class field theory, and complex multiplication}, publisher={Wiley}, date={1997}}

\bib{Diamond}{book} {author={ F. Diamond},  author={J. Shurman}, title={A first course in modular forms}, publisher={Springer}, date={2005}}

\bib{Forster}{book} {author={O. Forster}, title={Lectures on Riemann surfaces},  publisher={Springer-Verlag}, date={1981}}

\bib{Freitag}{book} {author={E. Freitag}, author={R. Busam}, title={Complex analysis}, publisher={Springer},  date={2009}}

\bib{Girondo}{book} {author={E. Girondo}, author={G. Gonz\'{a}lez-Diez}, title={Introduction to compact Riemann surfaces and dessins d'enfants}, publisher={Cambridge University Press}, date={2012}}

\bib{Grothendieck}{article} {author={A. Grothendieck}, title={Esquisse d'un programme}, date={1984}, publisher={manuscript}}

\bib{Hatcher}{book} {author={A. Hatcher}, title={Algebraic Topology}, publisher={Cambridge University Press}, date={2002}}

\bib{Ismail}{article} {author={M. E. H. Ismail}, author={C. Zhang}, title={Zeros of entire functions and a problem of Ramanujan}, journal={Adv. Math.}, volume={209}, date={2007}, pages={363--380}}

\bib{Jones}{book} {author={G. A. Jones},  author={J. Wolfart}, title={Dessins d'enfants on Riemann surfaces}, publisher={Springer},  date={2016}}

\bib{Zvonkin}{book} {author={S. K. Lando},  author={A. K. Zvonkin}, title={Graphs on surfaces and their applications}, publisher={Springer}, date={2004}}

\bib{Lawden}{book} {author={D. F. Lawden}, title={Elliptic functions and applications}, publisher={Springer-Verlag}, date={1989}}

\bib{Lehto}{book} {author={O. Lehto. K. I}, author={Virtanen}, title={Quasiconformal mappings in the plane}, publisher={Springer-Verlag}, date={1973}}

\bib{Lutovac}{book} {author={M. D. Lutovac}, author={D. V. Tosic}, author={B. L. Evans}, title={Filter Design for Signal Processing
Using MATLAB and Mathematica}, publisher={Prentice Hall}, date={2001}}

\bib{Maskit}{article} {author={B. Maskit}, title={Canonical domains on Riemann surfaces}, journal={Proc. Amer. Math. Soc.} volume={106(3)}, date={1989}, pages={713--721}}

\bib{Tsang}{article} {author={T. W. Ng}, author={C. Y. Tsang}, title={Chebyshev-Blaschke products: solutions to certain approximation problems and differential equations}, journal={J. Comp. and App. Math.} volume={277}, date={2015},  pages={106--114}}

\bib{Tsang1}{article} {author={T. W. Ng}, author={C. Y. Tsang}, title={Polynomials versus finite Blaschke products, in Blaschke products and their applications}, journal={Fields institute communications}, volume={65}, date={2013}, pages={249--273}}

\bib{WangNg}{article} {author={T. W. Ng}, author={M. X. Wang},  title={Ritt's theory on the unit disk},  journal={Forum Mathematicum}, volume={25(4)} date={2013}, pages={821--851}}

\bib{Remmert}{book} {author={R. Remmert}, title={Classical topics in complex function theory}, publisher={Springer-Verlag}, date={1998}}

\bib{Big rudin}{book} {author={W. Rudin}, title={Real and complex analysis}, publisher={McGraw-Hill},  date={1970}}

\bib{Schneps}{article} {author={L. Schneps}, title={Dessins d'enfants on the Riemann sphere, in The Grothendieck theory of dessins d'enfants.} journal={London Math. Soc. Lecture Notes Series},  volume={200}, pages={47--77}, date={1994}}

\bib{Shabat}{book} {author={G.B. Shabat}, author={V. Voevodsky}, title={Drawing curves over number fields, in The Grothendieck Festschrift, vol. III}, publisher={Birkhauser}, date={1990}}

\bib{Wang}{article} {author={M. X. Wang}, title={Factorizations of finite mappings on Riemann surfaces}, publisher={M. Phil thesis, the University of Hong Kong},  date={2008}}

\bib{Whittaker}{book} {author={E. T. Whittaker}, author={G. N. Watson}, title={A course of modern analysis}, publisher={Cambridge University Press}, date={1927}}

\end{biblist}
\end{bibdiv}
\end{document}